\documentclass{article}

\usepackage{microtype}
\usepackage{graphicx}
\usepackage{subfigure}
\usepackage{hyperref}

\usepackage{amsthm}
\usepackage{amsmath}
\usepackage{amssymb}
\usepackage{graphicx}
\usepackage{color}
\usepackage{ifpdf}
\usepackage{url}
\usepackage{hyperref}
\usepackage{authblk}
\usepackage[usenames,dvipsnames]{xcolor}
\usepackage{paralist}
\usepackage{algorithm}
\usepackage{algpseudocode}
\usepackage{algorithmicx}
\usepackage[T1]{fontenc} % Use 8-bit encoding that has 256 glyphs
\usepackage[hmarginratio=1:1, top=32mm, columnsep=20pt, margin=2.5cm]{geometry} % Document margins
\usepackage{multicol} % Used for the two-column layout of the document
\usepackage[hang, small,labelfont=bf,up,textfont=it,up]{caption} % Custom captions under/above floats in tables or figures
\usepackage{booktabs} % Horizontal rules in tables
\usepackage{float} % Required for tables and figures in the multi-column environment - they need to be placed in specific locations with the [H] (e.g. \begin{table}[H])
\usepackage{hyperref} % For hyperlinks in the PDF

\theoremstyle{plain}
\newtheorem{theorem}{Theorem}[section]

\newtheorem{lemma}{Lemma}[section]

\theoremstyle{definition}
\newtheorem{definition}{Definition}[section]
\newtheorem{assumption}{Assumption}[section]
\newtheorem{remark}{Remark}[section]
\usepackage{todonotes}

\newcommand{\R}{\mathbb{R}}
\newcommand{\Rext}{\R\cup\{+\infty\}}

\newcommand{\set}[1]{\left\{#1\right\}}

\newcommand{\norm}[1]{\left\Vert#1\right\Vert}
\newcommand{\norms}[1]{\Vert#1\Vert}

\newcommand{\Eproof}{\hfill $\square$}

\newcommand{\proj}{\mathrm{proj}}

\newcommand{\argmin}{\mathrm{arg}\!\displaystyle\min}

\newcommand{\dom}[1]{\mathrm{dom}(#1)}

\newcommand{\zero}[1]{{\boldsymbol{0}}}

\newcommand{\Xc}{\mathcal{X}}

\newcommand{\Lc}{\mathcal{L}}

\newcommand{\Tc}{\mathcal{T}}

\newcommand{\Nc}{\mathcal{N}}
\newcommand{\Oc}{\mathcal{O}}

\newcommand{\iprod}[1]{\left\langle #1\right\rangle}
\newcommand{\iprods}[1]{\langle #1\rangle}

\newcommand{\BigO}[1]{\mathcal{O}\left(#1\right)}

\newcommand{\beforesubsec}{\vspace{-1.5ex}}
\newcommand{\aftersubsec}{\vspace{-1ex}}
\newcommand{\beforesec}{\vspace{-1.25ex}}
\newcommand{\aftersec}{\vspace{-1ex}}

\title{A Newton Frank-Wolfe Method for Constrained Self-Concordant Minimization}

\author[1]{Deyi Liu}
\author[2]{Volkan Cevher}
\author[1]{Quoc Tran-Dinh}
\affil[1]{Department of Statistics and Operations Research\\
The University of North Carolina at Chapel Hill\\
318 Hanes Hall, Chapel Hill, NC 27599-3260\\
\newline \url{quoctd@email.unc.edu}, \url{deyiliu@live.unc.edu}.}
\affil[2]{Laboratory for Information and Inference Systems, EPFL, Lausanne, Switzerland
\newline \url{volkan.cevher@epfl.ch}.}

\date{}
\begin{document}
\maketitle

\begin{abstract}
\normalsize
We demonstrate how to scalably solve a class of constrained self-concordant minimization problems using linear minimization oracles (LMO) over the constraint set. We prove that the number of LMO calls of our method is nearly the same as that of the Frank-Wolfe method in $L$-smooth case. Specifically, our Newton Frank-Wolfe method uses $\BigO{\varepsilon^{-\nu}}$ LMO's, where $\varepsilon$ is the desired accuracy and $\nu :=  1 + o(1)$. 
In addition, we demonstrate how our algorithm can exploit the improved variants of the LMO-based schemes, including away-steps, to attain linear convergence rates. 
%Here, $\nu$ depends on the choice of parameters in our algorithm and can be made arbitrarily close to $1$. 
%Our algorithm requires a small number of iterations while exploiting advantages of LO-based schemes, including away-step with linear rates.
We also provide numerical evidence with portfolio design with the competitive ratio, D-optimal experimental design, and logistic regression with the elastic net where Newton Frank-Wolfe outperforms the state-of-the-art. 

\end{abstract}

%%% Main body.
%%%%%%%%%%%%%%%%%%%%%%%%%%%%%%%%%%%%%%%%%%%
%%% 1. Introduction.
%%%%%%%%%%%%%%%%%%%%%%%%%%%%%%%%%%%%%%%%%%%
\vspace{-2ex}
\beforesec
\section{Introduction}\label{sec:intro}
\aftersec
In this paper, we consider the following constrained convex optimization problem:
\begin{equation}\label{eq:constr_cvx}
f^{\star} := \min_{x\in\Xc}f(x). 
\end{equation}
Among the first-order methods, the Frank-Wolfe (FW), aka conditional gradient method  \cite{Frank1956} has gained tremendous popularity lately due to its scalability and its theoretical guarantees when the objective is $L$-smooth. 

The scalability of FW is mainly due to its main computational primitive, called the linear minimization oracle (LMO):
\begin{equation}\label{eq:lin_oracle}
\Lc_{\Xc}(s) := \argmin_{u\in\Xc} \iprods{s, u}.
\end{equation}
There are many applications, such as latent group lasso and simplex optimization problems where computing the LMO is significantly cheaper as compared to projecting onto the set $\Xc$. 

However, there are many machine learning problems where the objective function involves logarithmic, ridge regularized exponential, and log-determinant functions. These problems so far cannot exploit the rate as well as the scalability of the FW algorithm or its key variants. 

Our work precisely bridges this gap by focusing on objective functions where $f : \R^p\to\R$ is standard self-concordant (see Definition~\ref{de:self_con_def}) and $\Xc$ is a nonempty, compact and convex set in $\R^p$.  We emphasize that the class of self-concordant functions intersects with the class of Lipschitz continuous gradient functions, but they are different. In particular, we assume 
\begin{assumption}\label{ass:self_concor_Lips}
The solution set $\Xc^{\star}$ of \eqref{eq:constr_cvx} is nonempty.
The function $f$ in \eqref{eq:constr_cvx} is standard self-concordant and its Hessian $\nabla^2 f(x)$ is nondegenerate\footnote{This condition can be relaxed to the case $f(x) := g(Ax)$ where $g$ is self-concordant with nondegenerate Hessian, but $f$ may have degenerate Hessian.} for any $x \in \dom{f}$.
$\Xc$ is compact and its LMO defined by \eqref{eq:lin_oracle} can be computed efficiently and accurately.
\end{assumption}

Under Assumption~\ref{ass:self_concor_Lips}, problem \eqref{eq:constr_cvx} covers various applications in  statistical learning and machine learning such as D-optimal design \cite{harman2009approximate,lu2013computing}, minimum-volume enclosing ellipsoid \cite{damla2008linear}, quantum tomography \cite{gross2010quantum}, logistic regression with elastic-net regularization \cite{TranDinh2017d}, portfolio optimization \cite{ryu2014stochastic}, and optimal transport \cite{peyre2019computational}. %, just to name a few. 

\vspace{0.5ex}
\noindent\textbf{Related work:}
Motivated by the fact that, for many convex sets, including simplex, polytopes, and spectrahedron, computing a linear minimization oracle is much more efficient than evaluating their projection \cite{Hazan2008, Jaggi2013}, various linear oracle-based algorithms have been proposed, see, e.g., \cite{Frank1956, Hazan2008, Jaggi2013,lan2016conditional,lan2016accelerated}. Recently, such approaches are extended to the primal-dual setting in \cite{Yurtsever2015,yurtsever2019conditional}. %,alphomotopycgm,alpCGAL}. 

The most classical one is the Frank-Wolfe algorithm proposed in \cite{Frank1956} for minimizing a quadratic function over a polytope.
It has been shown that the convergence rate of this method is  $\BigO{1/t}$ and is tight under the $L$-smoothness assumption, where $t$ is the iteration counter.

Many recent papers have attempted to improve the convergence rate of the Frank-Wolfe algorithm and its variants by imposing further assumptions or exploiting the underlying problem structures.
For instance, \cite{Beck2004} showed a linear convergence of the Frank-Wolfe method under the assumption that $f$ is a quadratic function and the optimal solution $x^{\star}$ is in the interior of $\Xc$.
\cite{guelat1986some} firstly proposed a variant of the Frank-Wolfe method with away-step and proved its linear rate to the optimal value if $f$ is strongly convex, $\Xc$ is a polytope, and the optimal solution $x^{\star}$ is in the interior of $\Xc$.

Recently, \cite{garber2013linearly} and \cite{lacoste2015global} showed that the result of \cite{guelat1986some}  still holds even $x^{\star}$ is on the boundary of $\Xc$.
This can be viewed as the first general global linear convergence result of Frank-Wolfe algorithms.
\cite{garber2015faster} showed that the convergence rate of the Frank-Wolfe algorithm can be accelerated up to $\BigO{1/t^2}$ if $f$ is strongly convex and $\Xc$ is a ``strongly convex set'' (see their definition).

All the results mentioned above rely on the $L$-smooth assumption of the objective function $f$. Moreover, the primal-dual methods \cite{Yurtsever2015,yurtsever2019conditional} suffer in convergence rate since they can only handle the self-concordant function by splitting the objective and then relying on the proximal operator of the self-concordant function.

For the non-$L$-smooth case, the literature is minimal. Notably, \cite{odor2016frank} is the first work, to our knowledge, that proved that the Frank-Wolfe method could converge with $\BigO{1/t}$ rate for the Poisson phase retrieval problem where $f$ is a logarithmic objective. This result relies on a specific simplex structure of the feasible set $\Xc$ and proved that the objective function $f$ is eventually $L$-smooth on $\Xc$.

In addition, \cite{damla2008linear} showed a linear convergence of the Frank-Wolfe method with away-step for the minimum-volume enclosing ellipsoid problem with a log-determinant objective. The algorithms and the analyses in the respective papers exploit the cost function and the structure, but it is not clear how they can handle more general self-concordant objectives. Note that since both objective functions in the aforementioned works are self-concordant, they are covered by our framework in this paper.

\vspace{0.5ex}
\noindent\textbf{Our approach and contribution:}
Our first goal is to tackle an important class of problems \eqref{eq:constr_cvx}, where existing LMO-based methods do not have convergence guarantees. Our results make sense when computing the LMO is cheaper than computing projections. Otherwise, the first-order methods of \cite{Tran-Dinh2013a} can also be applied.
%However, we still assume that $\Xc$ has a tractable linear oracle defined by \eqref{eq:lin_oracle} that can be worth to exploit. 

For this purpose, we apply a projected Newton method to solve \eqref{eq:constr_cvx} and use the Frank-Wolfe method in the subproblems to approximate the projected Newton direction. This approach leads to a double-loop algorithm, where the outer loop performs an inexact projected Newton scheme, and the inner loop carries out an adaptive Frank-Wolfe method.

Notice that our algorithm enjoys several additional computational advantages.
When the feasible set $\Xc$ is a polytope, our subproblem becomes minimizing a quadratic function over a polytope.
By the result of \cite{lacoste2015global}, we can use the Frank-Wolfe algorithm with away-steps to attain linear convergence.
Since in the subproblem our objective function is quadratic, the optimal step size at each iteration has a closed-form expression, leading to structure exploiting variants (see Algorithm \ref{alg:A2}).
Finally, our algorithm can enhance Frank-Wolfe-type approaches by using the inexact projected Newton direction. 

To this end, our contribution can be summarized as follows:

\begin{compactitem}
\item[(a)] We propose a double-loop algorithm to solve \eqref{eq:constr_cvx} when $f$ is \textit{self-concordant} (see Definition~\ref{de:self_con_def}) and $\Xc$ is equipped with a \textit{tractable} linear oracle.
The proposed algorithm is self-adaptive, i.e. it does not require tuning for the step-size and accuracy of the subproblem.

\item[(b)] We prove that the gradient and Hessian complexity of our method is $\BigO{\ln\left(\frac{1}{\epsilon}\right)}$, while the LMO complexity is $\BigO{\frac{1}{\epsilon^{\nu}}}$ where $\nu = 1 + \frac{\ln(1-2\beta)}{\ln(\sigma)}$ for some constants $\sigma > 0$ and $\beta > 0$.
When $\beta$ approaches zero, the complexity bound $\BigO{\frac{1}{\epsilon^{\nu}}}$  also approaches $\BigO{\frac{1}{\epsilon}}$ as in the Frank-Wolfe methods for the $L$-smooth case.
\end{compactitem}
To our knowledge, this work is the first in studying LMO-based methods for solving \eqref{eq:constr_cvx} with non-Lipschitz continuous gradient functions on a general convex set $\Xc$.
It also covers the models in \cite{damla2008linear} and \cite{odor2016frank} as special cases, via a completely different approach.

\vspace{0.5ex}
\noindent\textbf{Paper outline:}
The rest of this paper is organized as follows.
Section \ref{sec:prelim_results} recalls some basic notation and preliminaries of self-concordant functions.
Section \ref{sec:main_alg} presents the main algorithm.
Section \ref{sec:convergence_analysis} proves the local linear convergence of our algorithm and gives a rigorous analysis of the total oracle complexity.
Three numerical experiments are given in Section~\ref{sec:experiment}.
Finally, we give the conclusion in Section \ref{sec:conc}.
All the technical proofs are deferred to the supplementary document (Supp. Doc).

%%%%%%%%%%%%%%%%%%%%%%%%%%%%%%%%%%%%%%%%%%%
%%%% 2. Basic assumptions and preliminary results
%%%%%%%%%%%%%%%%%%%%%%%%%%%%%%%%%%%%%%%%%%%
%\vspace{-2ex}
\beforesec
\section{Theoretical Background}\label{sec:prelim_results}
\aftersec

%\beforesubsec
\noindent\textbf{Basic notation:}
We work with Euclidean spaces, $\R^p$ and $\R^n$, equipped with standard inner product $\iprods{\cdot,\cdot}$ and Euclidean norm $\norm{\cdot}$.
For a given proper, closed, and convex function $f : \R^p\to\Rext$, $\dom{f} := \set{x\in\R^p \mid f(x) < +\infty}$ denotes the domain of $f$, $\partial{f}$ denotes the subdifferential of $f$, and $f^{\ast}$ is its Fenchel conjugate.
For a symmetric matrix $A\in\R^{n\times n}$, $\lambda_{\max}(A)$ denotes the largest eigenvalue of $A$.  We use $[k]$ to denote the set $\set{1,\cdots,k}$, and $e$ to denote the vector whose elements are $1$s.  
For a vector $u\in \R^p$, $\mathrm{Diag}(u)$ is a $p\times p$ diagonal matrix formed by $u$. 
We also define two nonnegative and monotonically increasing functions $\omega(\tau) := \tau - \ln(1 + \tau)$ for $\tau \in [0, \infty)$ and $\omega_{\ast}(\tau) := -\tau - \ln(1 - \tau)$ for $\tau \in [0,1)$.

\beforesubsec
\subsection{Self-concordant Functions}\label{subsec:self_cond_def}
\aftersubsec
%Our class of objective functions used in \eqref{eq:constr_cvx} is self-concordant introduced in \cite{Nesterov1994}, which is defined as follows:
We recall the definition of self-concordant functions introduced in \cite{Nesterov1994} here.

%%% Definition 2.1.
\begin{definition}\label{de:self_con_def}
A three-time continuously differentiable univariate function $\varphi : \R\to\R$ is said to be self-concordant with a parameter $M_{\varphi} \geq 0$ if $\vert\varphi{'''}(\tau)\vert \leq M_{\varphi}\varphi{''}(\tau)^{3/2}$ for all $\tau\in\dom{\varphi}$.
A three-time continuously differentiable function $f : \R^p\to\R$ is said to be self-concordant with a parameter $M_f \geq 0$ if $\varphi(\tau) := f(x + \tau v)$ is self-concordant with the same parameter $M_f$ for any $x\in\dom{f}$ and $v\in\R^p$.
If $M_f = 2$, then we say that $f$ is standard self-concordant.
\end{definition}

Note that any self-concordant function $f$ can be rescaled to the standard form as $\hat{f} := \frac{M_f^2}{4}f$.
When $\dom{f}$ does not contain straight line, $\nabla^2{f}(x)$ is positive definite, and therefore we can define a local norm associated with $f$ together with its dual norm as follows:
\begin{equation*}
\left\{
\begin{array}{l}
\norms{u}_{x} := \big(u^{\top}\nabla^2f(x)u\big)^{1/2}, \\
\norms{u}_{x}^{\ast} := \big(u^{\top}\nabla^2f(x)^{-1}u\big)^{1/2}.
\end{array}
\right.
\end{equation*}
These norms are weighted and satisfy the Cauchy-Schwarz inequality $\iprods{u, v} \leq \norms{u}_x\norms{v}_x^{\ast}$ for $u,v\in\R^p$.

The class of self-concordant functions is sufficiently broad to cover many applications.
It is closed under nonnegative combination and linear transformation.
Any linear and convex quadratic functions are self-concordant.
The function $f_1(x) := -\log(x)$ and $f_2(x) := x\log(x) - \log(x)$ are self-concordant.
For matrices, $f_3(X) := -\log\det(X)$ is also self-concordant, which is widely used in covariance estimation-type problems.
In statistical learning, the regularized logistic regression model with $f_4(x) := \frac{1}{n}\sum_{i=1}^n\log(1 + \exp(-y_ia_i^{\top}x)) + \frac{\mu_f}{2}\norms{x}^2$ and the regularized Poisson regression model with $f_5(x) := \frac{1}{n}\sum_{i=1}^n\left(y_i\exp(\frac{-a_i^{\top}x}{2}) + \exp(\frac{a_i^{\top}x}{2})\right) + \frac{\mu_f}{2}\norms{x}^2$ are both self-concordant.
Note that all the functions introduced above are not $L$-smooth except for $f_4$.
In addition, any strongly convex function with Lipschitz Hessian continuity is also self-concordant.
See \cite{SunTran2017gsc, Ostrovskii2018} for more examples and theoretical results.

\beforesubsec
\subsection{Approximate Solutions}\label{subsec:approx_sols}
\aftersubsec
Since the Hessian $\nabla^2 f(x)$ is nondegenerate, \eqref{eq:constr_cvx} has only one optimal solution $x^{\star}$. 
Moreover, $\nabla^2{f}(x^{\star}) \succ 0$.
Our goal is to design an algorithm to approximate $x^{\star}$ as follows:

%%% Definition 3.1.
\begin{definition}\label{def:inexact_sol}
Given a tolerance $\varepsilon > 0$, we say that $x_{\varepsilon}^{\star}$ is an $\varepsilon$-solution of \eqref{eq:constr_cvx} if
\begin{equation}\label{def:inexact_sol_prob}
\norms{x_{\varepsilon}^{\star} - x^{\star}}_{x^{\star}} \leq \varepsilon.
\end{equation}
\end{definition}
Different from existing Frank-Wolfe methods where an approximate solution $x_{\varepsilon}^{\star}$ is defined by $f(x_{\varepsilon}^{\star}) - f^{\star} \leq \varepsilon$, we define it via a local norm. 
However, we show in Theorem \ref{tm:complex_analysis_obj_value} that these two concepts are related to each other. 

%%%%%%%%%%%%%%%%%%%%%%%%%%%%%%%%%%%%%%%%%%%
%%%%% 3. Conditional Gradient-Based Algorithm.
%%%%%%%%%%%%%%%%%%%%%%%%%%%%%%%%%%%%%%%%%%%
\beforesec
\section{The Proposed Algorithm}\label{sec:main_alg}
\aftersec
Since $f$ in \eqref{eq:constr_cvx} is standard self-concordant, we first approximate it by a quadratic surrogate and apply a projected Newton method to solve \eqref{eq:constr_cvx}.

More precisely, given $x\in\dom{f}\cap\Xc$, the projected Newton method computes a search direction at $x$ by solving the following constrained convex quadratic program:
\begin{equation}\label{eq:sub_problem}
\begin{array}{ll}
T(x) := \displaystyle\argmin_{u\in\Xc}\Big\{ Q_f(u; x) := \iprod{\nabla f(x), u - x} + \frac{1}{2}(u - x)^{\top}\nabla^2f(x)(u - x) \Big\}.
\end{array}
\end{equation}
Since $\nabla^2{f}(x)$ is positive definite by Assumption~\ref{ass:self_concor_Lips}, $T(x)$ is the unique solution of \eqref{eq:sub_problem}.
However, this problem often does not have a closed-form solution, and we need to approximate it up to a given accuracy.
Since we aim at exploiting  LMO of $\Xc$, we apply a Frank-Wolfe scheme to solve \eqref{eq:sub_problem}.
The optimality condition of \eqref{eq:sub_problem} becomes
\begin{equation}\label{eq:opt_cond1}
\iprods{\nabla Q_f(T(x);x), T(x) - u} \leq 0,~~\forall u\in\Xc,
\end{equation}
where $\nabla Q_f(T(x);x) = \nabla{f}(x) + \nabla^2{f}(x)(T(x) - x)$.
Using this optimality condition, we can define an inexact solution of \eqref{eq:sub_problem} as follows:

%%% Definition 3.1.
\begin{definition}\label{def:inexact_sol_sub_prob}
Given a tolerance $\eta > 0$, we say that $T_{\eta}(x)$ is an $\eta$-solution of \eqref{eq:sub_problem} if
\begin{equation}\label{eq:inexact_sol_sub_prob}
\max_{u\in\Xc}\iprods{\nabla Q_f(T_{\eta}(x);x), T_{\eta}(x) - u} \leq \eta^2.
\end{equation}
\end{definition}
The following lemma, whose proof is in Supp. Doc. \ref{apdx:lm:inexact_sol_prop}, shows that the distance between $T(x)$ and $T_{\eta}(x)$ can be bounded by $\eta$. 
Therefore, this justifies the well-definedness of Definition \ref{def:inexact_sol_sub_prob}.

%% Lemma 3.1.
\begin{lemma}\label{lm:inexact_sol_prop}
Let $T_{\eta}(x)$ be an $\eta$-solution defined by Definition \ref{def:inexact_sol_sub_prob} and $T(x)$ be the exact solution of \eqref{eq:sub_problem}. 
Then, it holds that $\norms{T_{\eta}(x) - T(x)}_{x}  \leq \eta$.
\end{lemma}

Now, we combine our inexact projected Newton scheme and the well-known Frank-Wolfe algorithm to develop a new algorithm as presented in Algorithm~\ref{alg:A1}.

%We apply a projected Newton scheme to solve \eqref{eq:constr_cvx}.
%Since this method often has a small number of iterations, we only need a small number of gradient and Hessian evaluations.
%However, the subproblem \eqref{eq:sub_problem} does not have a closed-form solution and is often required to solve iteratively.
%To exploit the LO oracle of $\Xc$, we apply the well-known Frank-Wolfe method to solve it approximately.
%The full algorithm is described in  Algorithm \ref{alg: A1}.

%%%%%%%%%%%%%%%%%%%%%%%%%%%%%
%%%  Algorithm 1.
%%%%%%%%%%%%%%%%%%%%%%%%%%%%%
\begin{algorithm}[ht!]\caption{(\textit{FW-Based Projected Newton})}\label{alg:A1}
\begin{algorithmic}
\normalsize
\State\textbf{Inputs:} Input  $\varepsilon > 0$ and  $x^0\in\dom{f}\cap\Xc$.
%\STATE\hspace{2ex}$\bullet$~.
\State\hspace{2ex}$\bullet$~Choose  $(\beta, \sigma, C) > 0$ such that \eqref{eq:condition_para} holds.
\State\hspace{2ex}$\bullet$~Choose $C_1 \in (0, 0.5)$ and $\delta \in (0, 1)$.
\State\hspace{2ex}$\bullet$~Set $\lambda_{-1} := \frac{\beta}{\sigma}$ and $\eta_{0} := \min\{\frac{\beta}{C}, C_1h^{-1}(\beta)\}$, where 
\State\hspace{3.2ex}$h$ is defined in \eqref{eq:define_h}.
\For{$k := 0,1, \cdots$}
    \vspace{0.25ex}
    \State $z^{k} := \textbf{FW}(\nabla f(x^{k}), \nabla^2f(x^{k})[\cdot], x^{k},\eta_k^2).$
    \vspace{0.5ex}
    \State $d^k := z^k - x^k$ and $\gamma_k := \norms{d^k}_{x^k}$.
    \vspace{0.5ex}
    \If{$\gamma_k + \eta_k \leq h^{-1}(\beta)$ \textbf{or} $\lambda_{k-1} \leq \beta$}
        \vspace{0.5ex}
               \State $\lambda_k := \sigma\lambda_{k-1}$ and    $\eta_{k+1} := \sigma\eta_k$
               \vspace{0.5ex}
               \State $x^{k+1} := x^k + d^k$. (\textbf{full-step})
        \Else
            \State $\lambda_k := \lambda_{k-1}$ and    $\eta_{k+1} := \eta_k$.
        \vspace{0.5ex}
            \State $\alpha_k := \delta(\gamma_k^2 - \eta_k^2)/(\gamma_k^3 + \gamma_k^2 - \eta_k^2\gamma_k)$.
        \vspace{0.5ex}
            \State $x^{k+1} := x^k + \alpha_k d^k$. (\textbf{damped-step})
        \EndIf
        \If{$\lambda_k \leq \varepsilon$}
            \State \textbf{return} $x^{k+1}$
        \EndIf
\EndFor

\end{algorithmic}
\end{algorithm}
%%%%%%%%%%%%%%%%%%%%%%%%%%%%%
%%% End of the algorithm.
%%%%%%%%%%%%%%%%%%%%%%%%%%%%%

\begin{algorithm}[ht!]\caption{(\textit{Customized Frank-Wolfe Subroutine}){\!\!\!}}\label{alg:A2}
\begin{algorithmic} 
\State \textbf{FW}($h, H[\cdot], u^0,\eta$)
\For{$t := 0,1, \cdots T$}
      \State $g_t := h + H(u^t - u^0)$.
      \vspace{0.4ex}
      \State $v^t := \arg\max_{s\in\Xc}\iprod{g^t, u^t - s}$.
      \vspace{0.4ex}
      \State $V_t := \iprod{g^t, u^t - v^t}$.
      \If{$V_t > \eta$}
           \State $\delta_t := \norms{v^t - u^t}^2_H$ and  $\tau_t := \min\left\{1, V_t/\delta_t\right\}$.{\!\!\!}
         \vspace{0.4ex}
             \State $u^{t+1} := (1- \tau_t)u^t + \tau_tv^t$.
        \Else
           \State \textbf{return} $u^t$.
        \EndIf
\EndFor
\end{algorithmic}
\end{algorithm}

Let us make a few remarks on Algorithm~\ref{alg:A1}.

%\begin{compactitem}
%\item 
$\bullet$~\textbf{Discussion on structure:}
Algorithm \ref{alg:A1} integrates both damped-step and full-step inexact projected Newton schemes.
First, it performs the damped-step scheme to generate $\set{x^k}$ starting from an initial point $x^0$ that may be far from the optimal solution $x^{\star}$.
Then, once $\norms{x^k - x^{\star}}_{x^{\star}} \leq \beta$ is satisfied, it switches to the full-step scheme.
For the damped-step stage, we will show later that Algorithm \ref{alg:A1} only performs a finite number of iterations.

%\item 
$\bullet$~\textbf{Discussion on $\lambda_k$:}
The quantity $\lambda_k$ upper bounds $\norms{x^k - x^{\star}}_{x^{\star}}$.
In the damped-step stage, we keep $\lambda_k$ unchanged while in the full-step one, we decrease $\lambda_k$ by a factor of $\sigma \in (0,1)$.
Consequently, $\norms{x^k - x^{\star}}_{x^{\star}}$ will converge linearly to $0$ in the full-step stage (see Theorem \ref{tm:linear_convergence}).

%\item 
$\bullet$~\textbf{Discussion on $\eta_k$:}
The quantity $\eta_k$ is used to measure $\norms{T(x^k) - z^k}_{x^k}$ (see \eqref{eq:sub_problem} for the definition of $T(x^k)$).
In Algorithm \ref{alg:A1}, $z^k$ is calculated as 
\begin{equation}\label{alg:outer_iter_full_step}
z^{k} := \textbf{FW}(\nabla f(x^{k}), \nabla^2f(x^{k})[\cdot], x^{k},\eta_k^2),
\end{equation}
and can be viewed as an approximate solution of \eqref{eq:sub_problem} at $x = x^k$.
Therefore, $\eta_k$ measures the accuracy for solving the subproblem.
In the damped-step stage, we keep $\eta_k$ as a constant. 
In the full-step one, $\eta_k$ is decreased by a factor of $\sigma \in (0,1)$ at each iteration to guarantee that we get a more accurate projected Newton direction when the algorithm approaches the optimal solution $x^{\star}$.

%\item 
$\bullet$~\textbf{Discussion on} $\gamma_k + \eta_k \leq h^{-1}(\beta)$:
When $\gamma_k + \eta_k > h^{-1}(\beta)$, we use a damped-step scheme with the step-size $\alpha_k := \frac{\delta(\gamma_k^2 - \eta_k^2)}{\gamma_k^3 + \gamma_k^2 - \eta_k^2\gamma_k}$.
This step-size is derived from Lemma \ref{lm:one_iter_obj_value} in Supp. Doc., and is in $(0, 1)$.
Once $\gamma_k + \eta_k \leq h^{-1}(\beta)$ is satisfied, we move to the full-step stage and no longer use the damped-step one. 
In addition, from Lemma \ref{lm:bound_lambda_by_bar_lambda} in Supp. Doc., we can see that if $\gamma_k + \eta_k \leq h^{-1}(\beta)$, then we have $\norms{x^k - x^{\star}}_{x^{\star}} \leq \beta$, which means that we already find a good initial point for the full-step stage.

%\item 
$\bullet$~\textbf{Discussion on  the FW subroutine:}
The subroutine $\textbf{FW}(h,H[\cdot], u^0,\eta)$ is a customized Frank-Wolfe method to solve the following convex constrained quadratic program: 
\begin{equation*}
\min_{x \in \Xc} \big\{\psi(x): = \iprod{h, x - u^0} + \tfrac{1}{2}\iprod{H(x - u^0), x - u^0}\big\}.
\end{equation*}
The step size $\tau_t$ in \textbf{FW} is computed via the exact linesearch condition (see \cite{lan2016conditional} for further details):
\begin{equation*}
\tau_t := \argmin_{\alpha \in [0,1]} \left\{\psi(u^t + \alpha(v^t - u^t))\right\}.
\end{equation*}
%\item 
$\bullet$~\textbf{Discussion on $\nabla^2f(\cdot)$:} 
In practice, we do not need to evaluate the full Hessian $\nabla^2{f}(x^k)$. 
We only need to evaluate the matrix-vector operator $\nabla^2{f}(x^k)v$ for a given direction $v$.
Similarly, the computation of $\gamma_k$ does not incur significant cost.
Indeed, since we have already computed $\nabla^2{f}(x^k)d^k$ in the \textbf{FW}, computing $\gamma_k$ requires only one additional vector inner product $\iprods{\nabla^2{f}(x^k)d^k, d^k}$.
%\end{compactitem}

%%%%%%%%%%%%%%%%%%%%%%%%%%%%%%%%%%%%%%%%%%%
%%%%% 4. Convergence Analysis.
%%%%%%%%%%%%%%%%%%%%%%%%%%%%%%%%%%%%%%%%%%%
\beforesec
\section{Convergence and Complexity Analysis}\label{sec:convergence_analysis}
\aftersec
Our analysis closely follows the outline below: % the  outlined as follows:
\begin{compactitem}
\item Given $\beta\in (0,1)$, we show that we only need a finite number of damped-steps to reach $x^k$ such that  $\norms{x^k - x^{\star}}_{x^{\star}} \leq \beta$.
\item Once $\norms{x^k - x^{\star}}_{x^{\star}} \leq \beta$ is satisfied, we prove a linear convergence of the full-step projected Newton scheme.
\item We finally estimate the overall linear oracle (LMO) complexity of Algorithm~\ref{alg:A1}.
\end{compactitem}

\beforesubsec
\subsection{Finite Complexity of Damped-Step Stage}\label{subsec:finite_steps}
\aftersubsec

%%%%%%%%%%%%%%%%%%%%%%%%%%%%%%%%%%%%%%
%%%%%%%%%%%%%%% Lemma Begin
%%%%%%%%%%%%%%%%%%%%%%%%%%%%%%%%%%%%%%

% \begin{lemma}\label{lm:one_iter_obj_value}
% Let  $\gamma_k := \norms{z^k - x^k}_{x^k}$ be the local distance between $z^k$ to $x^k$, and  $\norms{z^k - T(x^k)}_{x^k} \leq \eta_k$. 
% If we choose $\alpha \in (0, 1)$ such that  $\alpha\gamma_k < 1$ and update $x^{k+1} := x^k + \alpha(z^k - x^k)$, then we have
% \begin{equation}\label{key_prop6}
% f(x^{k+1}) \leq f(x^{k}) - \left[\alpha( \gamma_k^2 - \eta_k^2) - \omega_{*}(\alpha\gamma_k)\right].
% \end{equation}
% If $\gamma_k > \eta_k$, then the optimal step size is $\alpha_k := \frac{\gamma_k^2 - \eta_k^2}{\gamma_k(\gamma_k^2 + \gamma_k - \eta_k^2)}$ which satisfies $\alpha_k\gamma_k < 1$. 
% Moreover, we have
% \begin{equation}\label{key_prop7}
% f(x^{k+1}) \leq f(x^{k}) - \omega\left(\frac{\gamma_k^2 - \eta_k^2}{\gamma_k}\right),
% \end{equation}
% where $\omega(\tau) := \tau - \ln(1 + \tau)$ and $\omega_{\ast}(\tau) := -\tau - \ln(1 - \tau)$ are two nonnegative and convex functions.
% \end{lemma}

%%%%%%%%%%%%%%%%%%%%%%%%%%%%%%%%%%%%%%
%%%%%%%%%%%%%%% End of Lemma 
%%%%%%%%%%%%%%%%%%%%%%%%%%%%%%%%%%%%%%

Before we present the main theorem of this section, let us first define a univariate function $h : \R_{+} \to \R_{+}$ as
\begin{equation}\label{eq:define_h}
h(\tau) := \frac{\tau(1 -2\tau + 2\tau^2)}{(1 - 2\tau)(1 - \tau)^2 - \tau^2}.
\end{equation}
The shape of $h$ is shown in Figure \ref{fig:h_and_condition}.
\begin{figure}[ht!]
\vspace{-0ex}
\begin{center}
\hspace{-0ex}\includegraphics[width=0.49\textwidth]{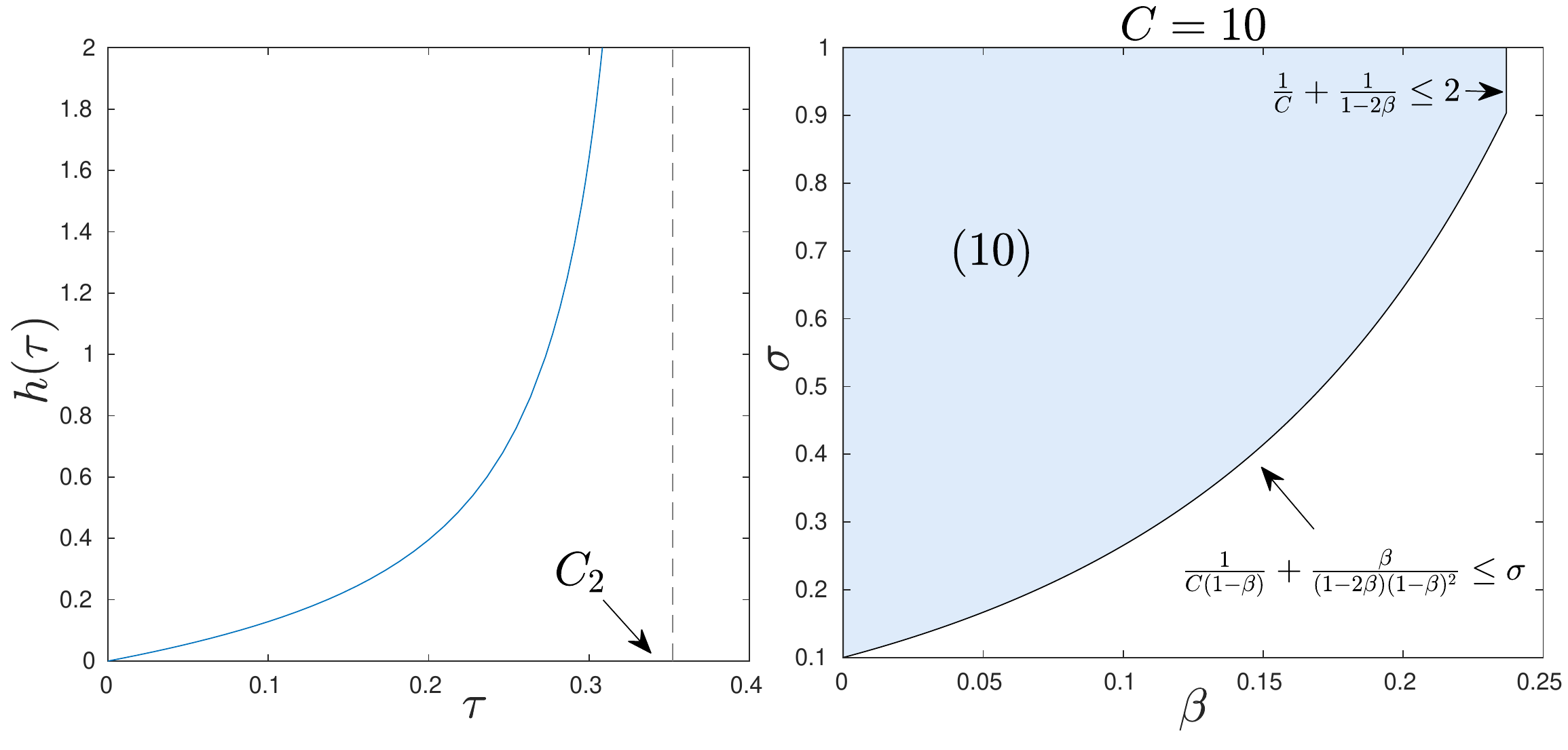}
\vspace{-0ex}
\caption{The shape of $h$ $($left$)$ and the feasible region of $(\beta,\sigma)$ for \eqref{eq:condition_para} when $C = 10$ $($right$)$.}
\label{fig:h_and_condition}
\end{center}  
\vspace{-0ex}
\end{figure} 

From Figure \ref{fig:h_and_condition}, $h$ is nonnegative and monotonically increasing on $[0, C_2)$ for the constant $C_2 \in (0.3, 0.4)$ such that $(1 - 2C_2)(1 - C_2)^2 - C_2^2 = 0$. 

%%%%%%%%%%%%%%%%%%%%%%%%%%%%%%%%%%%%%%
%%%%%%%%%%%%%%% Lemma Begin
%%%%%%%%%%%%%%%%%%%%%%%%%%%%%%%%%%%%%%
% \begin{lemma}\label{lm:bound_lambda_by_bar_lambda}
% Let  $\lambda_k := \norms{x^{k} - x^{\star}}_{x^{\star}}$, $\bar{\gamma}_k := \norms{T(x^k) - x^k}_{x^k}$, $\gamma_k := \norms{z^k - x^k}_{x^k}$, and $h$ be defined by \eqref{eq:define_h}.
% If $\gamma_k + \eta_k \in (0, C_2)$, then we have
% \begin{equation}\label{key_prop8}
% \lambda_k \leq h(\bar{\gamma}_k) \leq h(\gamma_k + \eta_k).
% \end{equation}
% \end{lemma}
%%%%%%%%%%%%%%%%%%%%%%%%%%%%%%%%%%%%%%
%%%%%%%%%%%%%%% End of Lemma 
%%%%%%%%%%%%%%%%%%%%%%%%%%%%%%%%%%%%%%

%%%%%%%%%%%%%%%%%%%%%%%%%%%%%%%%%%%%%%
%%%%%%%%%%%%%%% Theorem Begin 
%%%%%%%%%%%%%%%%%%%%%%%%%%%%%%%%%%%%%%
The following theorem, whose proof is in Supp. Doc. \ref{apdx:proof_finite_step_phase1}, states that Algorithm \ref{alg:A1} only needs a finite number of LMO calls $\Tc_1$ to achieve $x^k$ such that $\norms{x^k - x^{\star}}_{x^{\star}} \leq \beta$. 
Although $\Tc_1$ is independent of tolerance $\epsilon$, it depends on the pre-defined constants $(\beta, C_1)$ in the algorithm and the data structure of $(f, \Xc)$. 

\begin{theorem}\label{tm:finite_step_phase1}
Let $\omega(\tau) := \tau - \ln(1 + \tau)$. If we choose the parameters as in Algorithm \ref{alg:A1}, then after at most
\begin{equation}\label{eq:k_max}
K : = \frac{f(x^{0}) - f(x^{\star})}{\delta\omega\left(\frac{1 - 2C_1}{C_1}h^{-1}(\beta)\right)}
\end{equation}
outer iterations of the damped-step scheme, we can guarantee that $\gamma_k + \eta_k \leq h^{-1}(\beta)$ for some $k \in [K]$, which implies that $\norms{x^{k} - x^{\star}}_{x^{\star}} \leq \beta$. 
Moreover, the total number of LMO calls is at most 
\begin{equation*}
\Tc_1 := \frac{6D_{\Xc}^2\lambda_{\max}(\nabla^2 f(x^0))}{(C_1h^{-1}(\beta))^2}\frac{1 - (1-\delta)^{K+1}}{\delta(1-\delta)^K},
\end{equation*}
where $D_{\Xc} := \displaystyle\max_{x,y\in \Xc}\norms{x-y}$.
The number of gradient $\nabla{f}(x^k)$ and Hessian $\nabla^2{f}(x^k)$ evaluations is also $K$.
\end{theorem}
%%%%%%%%%%%%%%%%%%%%%%%%%%%%%%%%%%%%%%
%%%%%%%%%%%%%%% End of Theorem 
%%%%%%%%%%%%%%%%%%%%%%%%%%%%%%%%%%%%%%
\begin{remark}\label{le:choice_of_delta}
We  show in Subsection~\ref{subsec:trade_off} how to choose $\delta$ such that the number of LMO calls in Theorem~\ref{tm:finite_step_phase1} is dominated by the one in the full-step stage in Theorem~\ref{tm:complex_analysis}.
%Hence, the overall LMO complexity will depend on the full-step stage.
\end{remark}

\beforesubsec
\subsection{Linear Convergence of Full-Step Stage}\label{subsec:linear_rate_long_step}
\aftersubsec
Since Theorem~\ref{tm:finite_step_phase1} shows that we only need a finite number of damped-steps to obtain $x^k$ such that $\norms{x^k - x^{\star}}_{x^{\star}} \leq \beta$. 
Therefore, without loss of generality, we always assume that $\norms{x^0 - x^{\star}}_{x^{\star}} \leq \beta$ in the rest of this paper.
Using this assumption,  we analyze  convergence rate of $\set{x^k}$ to the unique optimal solution $x^{\star}$ of \eqref{eq:constr_cvx}. 
In this case, Algorithm \ref{alg:A1} always choose full-steps, i.e., $x^{k+1} := x^k + d^k = z^k$.

%%%%%%%%%%%%%%%%%%%%%%%%%%%%%%%%%%%%%%
%%%%%%%%%%%%%%% Lemma Begin
%%%%%%%%%%%%%%%%%%%%%%%%%%%%%%%%%%%%%%
% \begin{lemma}\label{lm:one_iter_ineq}
% Suppose that $\lambda_k := \norms{x^k - x^{\star}}_{x^{\star}} \leq \beta$, where $\beta \in (0, 1)$ is chosen by Algorithm \ref{alg:A1}. 
% Then, we have
% \begin{equation}\label{key_prop1}
% \lambda_{k+1} \leq \frac{\eta_k}{1-\lambda_k} + \frac{\lambda_k^2}{(1 - \lambda_k)^2(1-2\lambda_k)}.
% \end{equation}
% In addition, we can also bound $\norms{x^{k+1} - x^k}_{x^k}$ as follows:
% \begin{equation}\label{key_prop2}
% \norms{x^{k+1} - x^k}_{x^k} \leq \eta_k + \frac{\lambda_k^2}{(1-2\lambda_k)(1-\lambda_k)} + \frac{\lambda_k}{1 - \lambda_k}.
% \end{equation}
% \end{lemma}
%%%%%%%%%%%%%%%%%%%%%%%%%%%%%%%%%%%%%%
%%%%%%%%%%%%%%% End of Lemma 
%%%%%%%%%%%%%%%%%%%%%%%%%%%%%%%%%%%%%%

The following theorem states a linear convergence of $\norms{x^k - x^{\star}}_{x^{\star}}$ and $\norms{x^{k+1} - x^{k}}_{x^{k}}$. 
The convergence of $\norms{x^{k+1} - x^{k}}_{x^{k}}$ will be used in Theorem \ref{tm:complex_analysis} to bound $\set{\nabla^2 f(x^k)}$ which is key to our LMO complexity analysis. 
The proof can be found in Supp. Doc.~\ref{apdx:proof_linear_convergence}.

%%%%%%%%%%%%%%%%%%%%%%%%%%%%%%%%%%%%%%
%%%%%%%%%%%%%%% Theorem Begin 
%%%%%%%%%%%%%%%%%%%%%%%%%%%%%%%%%%%%%%
\begin{theorem}\label{tm:linear_convergence}
Suppose $\norms{x^0 - x^{\star}}_{x^{\star}} \leq \beta$ and the triple $(\sigma, \beta, C)$ satisfies the following conditions:
\begin{equation}\label{eq:condition_para}
\left\{\begin{array}{l}
  \sigma \in (0,1), ~~\beta \in (0, 0.5), ~~C > 1, \vspace{1ex}\\
  \frac{1}{C(1-\beta)} + \frac{\beta}{(1-2\beta)(1-\beta)^2}\leq \sigma, \vspace{1ex}\\
  \frac{1}{C} + \frac{1}{(1-2\beta)}\leq 2.
\end{array}\right.
\end{equation}
Let $\eta_k := \frac{\beta\sigma^k}{C}$ and $\{x^k\}$ be updated by the full-step scheme in Algorithm \ref{alg:A1}. 
Then, for $k\geq 0$, we have 
\begin{equation*}\label{key_prop3}
\norms{x^k - x^{\star}}_{x^{\star}} \leq \beta\sigma^k ~~\text{and}~~\norms{x^{k+1} - x^{k}}_{x^{k}} \leq 2\beta\sigma^k.
\end{equation*}
\end{theorem}
%%%%%%%%%%%%%%%%%%%%%%%%%%%%%%%%%%%%%%
%%%%%%%%%%%%%%% End of Theorem 
%%%%%%%%%%%%%%%%%%%%%%%%%%%%%%%%%%%%%%

Theorem~\ref{tm:linear_convergence} shows that $\set{x^k}$ linearly converges to $x^{\star}$ with a contraction factor $\sigma \in (0, 1)$ chosen from \eqref{eq:condition_para}.
Figure \ref{fig:h_and_condition} shows the feasible region of $(\beta, \sigma)$ for $\eqref{eq:condition_para}$ when $C = 10$.
From this figure, we can see that \eqref{eq:condition_para} will always hold once $\beta$ is sufficiently small. 
Therefore, theoretically, we can let $\beta$ arbitrarily close to $0$. 

%%%%%%%%%%%%%%%%%%%%%%%%%%%%%%%%%%%%%%%%%%%
%%%%% 5. LMO Oracle Complexity Analysis.
%%%%%%%%%%%%%%%%%%%%%%%%%%%%%%%%%%%%%%%%%%%
\beforesubsec
\subsection{Overall LMO Complexity Analysis}\label{subsec:lo_complexity}
\aftersubsec
%%%%%%%%%%%%%%%%%%%%%%%%%%%%%%%%%%%%%%
%%%%%%%%%%%%%%% Lemma Begin
%%%%%%%%%%%%%%%%%%%%%%%%%%%%%%%%%%%%%%
This subsection focuses on the analysis of LMO complexity of Algorithm~\ref{alg:A1}. 
We first show that Algorithm~\ref{alg:A1} needs $\BigO{\varepsilon^{-2\nu}}$ LMO calls to reach an $\varepsilon$-solution defined by \eqref{def:inexact_sol_prob} where $\nu := 1 + \frac{\ln(1-2\beta)}{\ln(\sigma)}$. 
Consequently,  we can show that it needs $\BigO{\varepsilon^{-\nu}}$-LMO calls  to find an $\varepsilon$-solution $x^{\star}_{\varepsilon}$ such that $f(x^{\star}_{\varepsilon}) - f^{\star} \leq \varepsilon$. 
See Supp. Doc. \ref{apdx:proof_complex_analysis} for details.

%%% Lemma 3.5.
% \begin{lemma}\label{lm:bound_sub_solver}
% At the $k$-th outer iteration of Algorithm~\ref{alg:A1}, if we run the Frank-Wolfe subroutine  \eqref{alg:outer_iter_full_step} to update $u^t$, then, after $T_k$ iterations, we have
% \begin{equation}\label{key_prop4}
% \min_{t = 1,\cdots, T_k} V_k(u^t) \leq \frac{6\lambda_{\max}(\nabla^2f(x^k))D_{\Xc}^2}{T_k + 1},
% \end{equation}
% where 
% $V_k(u^j) := \max_{u\in\Xc}\iprod{\nabla Q_f(u^j;x^k), u^j - u}$. 
% As a result, the number of LMO calls at the $k$-th outer iteration of Algorithm \ref{alg:A1} is at most $O_k := \frac{6\lambda_{\max}(\nabla^2f(x^k))D_{\Xc}^2}{\eta_k^2}$.
% \end{lemma}
%%%%%%%%%%%%%%%%%%%%%%%%%%%%%%%%%%%%%%
%%%%%%%%%%%%%%% End of Lemma 
%%%%%%%%%%%%%%%%%%%%%%%%%%%%%%%%%%%%%%

%%%%%%%%%%%%%%%%%%%%%%%%%%%%%%%%%%%%%%
%%%%%%%%%%%%%%% Theorem Begin 
%%%%%%%%%%%%%%%%%%%%%%%%%%%%%%%%%%%%%%
\begin{theorem}\label{tm:complex_analysis}
Suppose that $\norms{x^{0} - x^{\star}}_{x^{\star}} \leq \beta$. 
If we choose the parameters $\beta$, $\sigma$, $C$, and $\{\eta_k\}$ as in Theorem \ref{tm:linear_convergence} and update $\set{x^k}$ by the full-steps, then to obtain an $\varepsilon$-solution $x_{\epsilon}^{\star}$ defined by \eqref{def:inexact_sol_prob}, it requires 
\begin{equation*}
\left\{\begin{array}{ll}
\BigO{\ln(\varepsilon^{-1})}~~\text{gradient evaluations $\nabla{f}(x^k)$}, \vspace{1ex}\\
\BigO{\ln(\varepsilon^{-1})}~~\text{Hessian evaluations $\nabla^2{f}(x^k)$, and} \vspace{1ex}\\
\BigO{\varepsilon^{-2\nu}} ~~\text{LMO calls, with $\nu := 1 + \frac{\ln(1-2\beta)}{\ln(\sigma)}$}.
\end{array}\right.
\end{equation*}
\vspace{-1ex}
\end{theorem}
%%%%%%%%%%%%%%%%%%%%%%%%%%%%%%%%%%%%%%
%%%%%%%%%%%%%%% End of Theorem 
%%%%%%%%%%%%%%%%%%%%%%%%%%%%%%%%%%%%%%

From Theorem~\ref{tm:complex_analysis}, we can observe that a small value of $\beta$ gives a better oracle complexity bound, but increases the number of oracle calls in the damped-step stage.
Hence, we trade-off between the damped-step stage and the full-step stage. 
In practice, we do not recommend to choose an extremely small $\beta$ but some value in the range of $[0.01, 0.1]$. 

%%%%%%%%%%%%%%%%%%%%%%%%%%%%%%%%%%%%%%
%%%%%%%%%%%%%%% Lemma Begin
%%%%%%%%%%%%%%%%%%%%%%%%%%%%%%%%%%%%%%
% \begin{lemma}\label{lm:bound_obj_value}
% Let $\gamma_k := \norms{x^{k+1} - x^k}_{x^k} = \norms{z^{k} - x^k}_{x^k}$ and  $\lambda_k := \norms{x^{k} - x^{\star}}_{x^{\star}}$. 
% Suppose that $x^0\in\dom{f}\cap\Xc$. 
% If $0 < \gamma_k,\lambda_k, \lambda_{k+1} < 1$, then we have
% \begin{equation}\label{key_prop5}
% f(x^{k+1}) \leq f(x^{\star}) + \frac{\gamma_k^2(\gamma_k + \lambda_k)}{1 - \gamma_k} + \eta_k^2 + \omega_{\ast}(\lambda_{k+1}),    
% \end{equation}
% where $\omega_{\ast}(\tau) := -\tau - \ln(1-\tau)$.
% \end{lemma}
%%%%%%%%%%%%%%%%%%%%%%%%%%%%%%%%%%%%%%
%%%%%%%%%%%%%%% End of Lemma 
%%%%%%%%%%%%%%%%%%%%%%%%%%%%%%%%%%%%%%

Finally, the following theorem states the LMO complexity of Algorithm~\ref{alg:A1} on the objective residuals.
%%%%%%%%%%%%%%%%%%%%%%%%%%%%%%%%%%%%%%
%%%%%%%%%%%%%%% Theorem Begin 
%%%%%%%%%%%%%%%%%%%%%%%%%%%%%%%%%%%%%%

%%% Theorem 3.4.
\begin{theorem}\label{tm:complex_analysis_obj_value}
Suppose that $\norms{x^{0} - x^{\star}}_{x^{\star}} \leq \beta$. 
If we choose $\sigma$, $\beta$, $C$, and $\{\eta_k\}$ as in Theorem \ref{tm:linear_convergence} and update $\{x^k\}$ by the full-steps, then we have
\begin{equation*}
f(x^{k+1}) - f(x^{\star}) \leq \left(\frac{12\beta^3}{1-2\beta} + \frac{\beta^2}{C^2} + \beta^2\right)\sigma^{2k}.
\end{equation*}
Consequently, the total LMO complexity of Algorithm~\ref{alg:A1} to achieve an $\epsilon$-solution $x_{\epsilon}^{\star}$ such that $f(x_{\epsilon}^{\star}) - f^{\star} \leq \epsilon$ is $\BigO{\epsilon^{-\nu}}$, where $\nu := 1 + \frac{\ln(1-2\beta)}{\ln(\sigma)}$.
\end{theorem}
%%%%%%%%%%%%%%%%%%%%%%%%%%%%%%%%%%%%%%
%%%%%%%%%%%%%%% End of Theorem 
%%%%%%%%%%%%%%%%%%%%%%%%%%%%%%%%%%%%%%

As a concrete example, if we choose $C := 10$ and then the conditions \eqref{eq:condition_para} of Theorem~\ref{tm:linear_convergence} hold if we choose $(\beta, \sigma) = (0.05, 0.1668)$.
In this case, $\nu := 1 + \frac{\ln(1 - 2\beta)}{\ln(\sigma)} = 1.0588$ which is very close to $1$.
The proof of Theorems~\ref{tm:complex_analysis} and \ref{tm:complex_analysis_obj_value} can be found in Supp. Doc. \ref{apdx:proof_complex_analysis} and \ref{apdx:proof_complex_analysis_obj_value}.

\beforesubsec
\subsection{Complexity Trade-off Between Two Stages}\label{subsec:trade_off} 
\aftersubsec
%We provide the detail analysis of the LMO complexity trade-off between two stages as stated in Subsection~\ref{subsec:trade_off}.
Given a sufficiently small target accuracy $\varepsilon > 0$, our goal is to find $\delta\in(0,1)$ such that the LMO complexity $\Tc_2 := \BigO{\varepsilon^{-2\nu}}$ in Theorem \ref{tm:complex_analysis} dominates $\Tc_1$ in Theorem~\ref{tm:finite_step_phase1}.
Let us choose $\delta := \varepsilon$.
Then, the number of iterations $K$ of the damped-step stage in Theorem~\ref{tm:finite_step_phase1} is $K = \frac{R}{\varepsilon} = \BigO{\frac{1}{\varepsilon}}$, where $R := \frac{f(x^{0}) - f(x^{\star})}{\omega\left(\frac{1 - 2C_1}{C_1}h^{-1}(\beta)\right)}$ is a fixed constant.
Moreover, for sufficiently small $\varepsilon$, we have $(1-\delta)^K = (1-\varepsilon)^{\frac{R}{\varepsilon}} = \BigO{\frac{1}{e^R}}$.
Hence, by Theorem~\ref{tm:finite_step_phase1}, the total LMO calls of the damped-step stage can be bounded by
\begin{equation*}
\Tc_1 := \BigO{\frac{1}{\delta(1-\delta)^K}} = \BigO{\frac{e^R}{\varepsilon }} = \BigO{\frac{1}{\varepsilon}}.
\end{equation*}
We conclude that the LMO complexity $\Tc_2 := \BigO{\varepsilon^{-2\nu}}$ in the full-step stage dominates the one $\Tc_1 = \BigO{\varepsilon^{-1}}$ in the damped-step stage.
% Overall, the total complexity of Algorithm~\ref{alg:A1} is stated as in Theorem~\ref{tm:complex_analysis}.
% Alternatively, Theorem~\ref{tm:complex_analysis_obj_value} shows that the total LMO complexity of $\BigO{\epsilon^{-\nu}}$ to achieve $f(x^{\star}_{\epsilon}) - f^{\star} \leq \epsilon$ for a given accuracy $\epsilon > 0$.

%%%%%%%%%%%%%%%%%%%%%%%%%%%%%%%%%%%%%%%%%%%
%%%%% 6. Experiment.
%%%%%%%%%%%%%%%%%%%%%%%%%%%%%%%%%%%%%%%%%%%
\beforesec
\section{Numerical Experiments}\label{sec:experiment}
\aftersec
We provide three numerical examples to illustrate the performance of Algorithm~\ref{alg:A1}.
We emphasize that the objective function $f$ of these examples does not have Lipschitz continuous gradient.
Hence, existing Frank-Wolfe and projected gradient-based methods may not have theoretical guarantees.
In the following experiments, we implement Algorithms~\ref{alg:A1} in Matlab running on a Linux desktop with 3.6GHz Intel Core i7-7700 and 16Gb memory.
Our code is available online at \href{https://github.com/unc-optimization/FWPN}{\color{blue}https://github.com/unc-optimization/FWPN}.

\begin{figure*}[ht!]
%\vspace{-2ex}
\begin{center}
\includegraphics[width=1\textwidth]{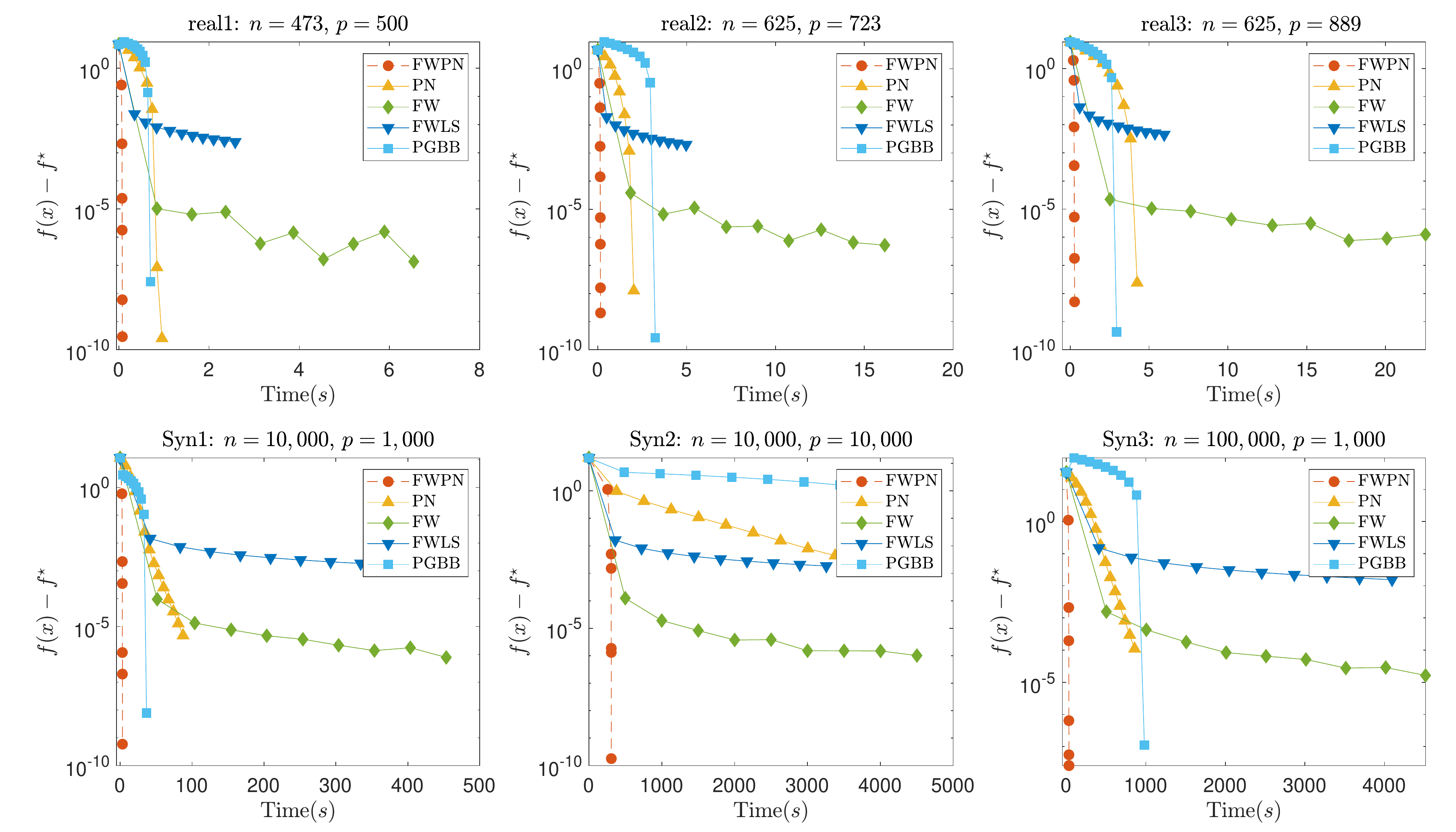}
\caption{A comparison between five methods for solving problem \eqref{eq:portfolio_opt} on six datasets.}
\label{fig:Port}
\end{center}  
\vspace{-2ex}
\end{figure*} 

\beforesubsec
\subsection{Portfolio Optimization}\label{subsec:portfolio_exam}
\aftersubsec
Consider the following portfolio optimization model widely studied in the literature, see e.g., \cite{ryu2014stochastic}:
\begin{equation}\label{eq:portfolio_opt}
\left\{\begin{array}{l}
{\displaystyle\min_{x\in\R^p}} f(x) := -\sum_{i=1}^n\ln(a_i^Tx) \vspace{1ex}\\
\text{s.t.}~~ \sum_{i=1}^px_i = 1, ~x\geq 0,
\end{array}\right.
\end{equation}
where $a_i \in\R^p$ for $i=1,\cdots, n$. 
Let $A  := [a_1,\cdots,a_n]^{\top} \in\R^{n\times p}$.
In the portfolio optimization model, $A_{ij}$ represents the return of stock $j$ in scenario $i$ and $\ln(\cdot)$ is the utility function. 
Our goal is to allocate assets to different stock companies to maximize the expected return.

We implement Algorithm \ref{alg:A1}, abbreviated by FWPN, to solve \eqref{eq:portfolio_opt}.
We also implement the standard projected Newton method which uses accelerated projected gradient method to compute the search direction, the Frank-Wolfe algorithm \cite{Frank1956} and its linesearch variant \cite{Jaggi2013}, and a projected gradient method using Barzilai-Borwein's step-size. 
We name these algorithms by PN, FW, FW-LS, and PG-BB, respectively.

We test these algorithms both on synthetic and real data.  
For the real data, we download three US stock datasets from \href{http://www.excelclout.com/historical-stock-prices-in-excel/}{http://www.excelclout.com/historical-stock-prices-in-excel/}. 
We name these datasets by \texttt{real1}, \texttt{real2}, and \texttt{real3}. 
We generate synthetic datasets as follows. 
We generate a matrix $A$ as $A := \text{ones}(n,p) + \Nc(0,0.1)$ which allows each stock to vary about 10\% among scenarios. 
We test with three examples, where$(n,p)=(10^4,10^3)$, $(10^4,10^4)$, and $(10^5,10^3)$, respectively. 
We call these three datasets \texttt{Syn1}, \texttt{Syn2}, and \texttt{Syn3}, respectively. 
The results and the performance of these five algorithms are shown in Figure \ref{fig:Port}.

From Figure \ref{fig:Port}, one can observe that our algorithm, FWPN, clearly outperforms the other competitors on both real and synthetic datasets. 
In our algorithm, we use a Frank-Wolfe method with away-step to solve the simplex constrained quadratic subproblem which has a linear convergence rate as proved in \cite{lacoste2015global}.
Both PGBB and PG work relatively well compared to other candidates.
As expected, the standard FW and its linesearch variant cannot reach a highly accurate solution.

\beforesubsec
\subsection{\texorpdfstring{$D$}~-Optimal Experimental Design}\label{subsec:D_optimal_design_exam}
\aftersubsec

\begin{figure*}[ht!]
\begin{center}
\includegraphics[width=1\textwidth]{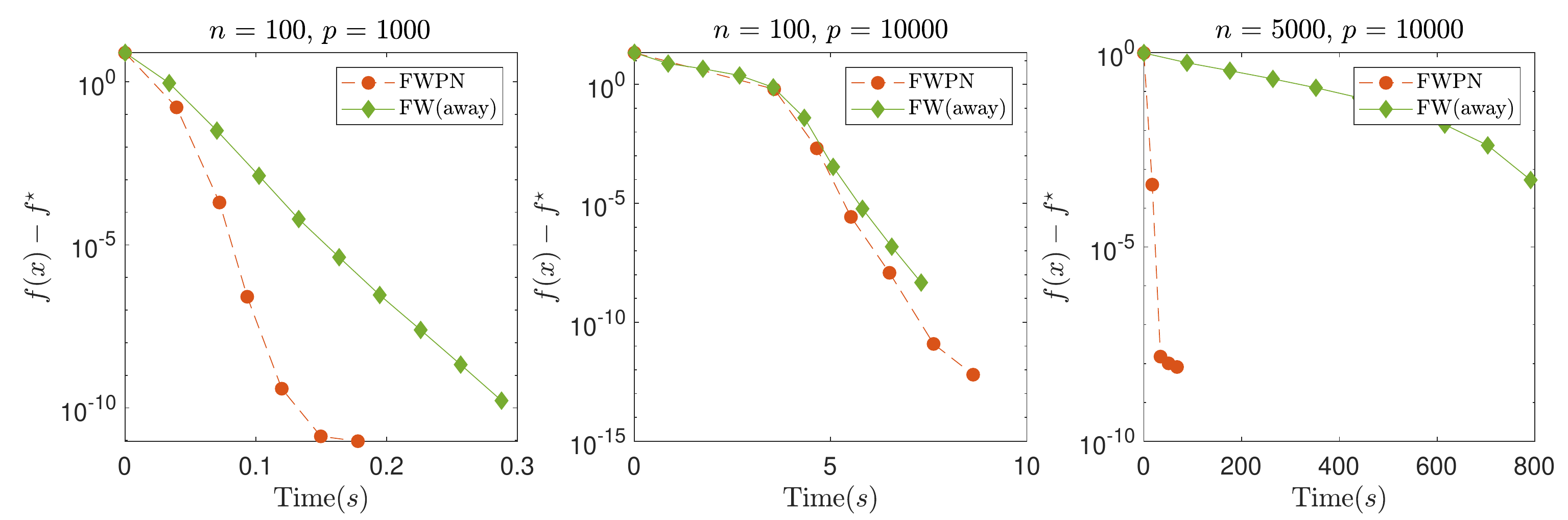}
\caption{A comparison between two algorithms for solving  \eqref{eq:D_optimal_design_exam} on three datasets.}
\label{fig:MVEE}
\end{center}  
\vspace{-2ex}
\end{figure*}

\begin{figure*}[ht!]
%\vspace{-2ex}
\begin{center}
\includegraphics[width=1\textwidth]{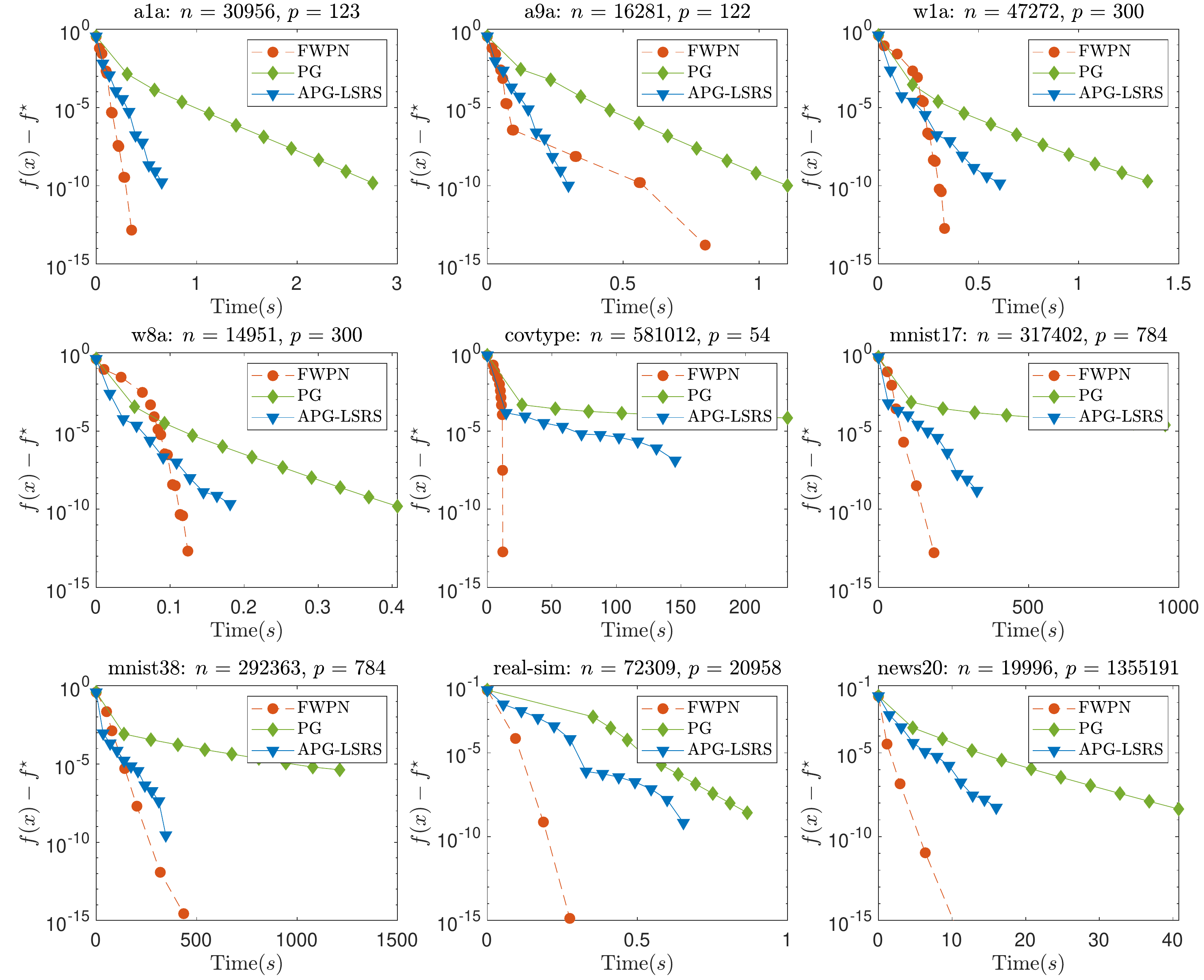}
\caption{A comparison between three methods for solving \eqref{eq:logistic_reg_exam1} on $9$ different datasets.}
\label{fig:logistic}
\end{center}  
\vspace{-3ex}
\end{figure*}  

The second example is the following convex optimization model in $D$-optimal experimental design:
\begin{equation}\label{eq:D_optimal_design_exam}
\left\{\begin{array}{ll}
{\displaystyle\min_{x\in \R^p}} &f(x) := -\log\det(AXA^{\top}) \vspace{1ex}\\
\text{s.t.}      &\sum_{j=1}^px_j = 1, ~x \geq 0,
\end{array}\right.
\end{equation}
where $A := [a_1,\cdots,a_p] \in \mathbb{R}^{n\times p}$, $X := \mathrm{Diag}(x)$, and $a_i \in \R^n$ for $i = 1,\cdots,p$.
It is well-known that the dual problem of \eqref{eq:D_optimal_design_exam} is the minimum-volume enclosing ellipsoid (MVEE) problem:
\begin{equation}\label{eq:MVEE_exam}
\left\{\begin{array}{ll}
{\displaystyle\min_{H \succeq 0}} & g(H) := -\log\det(H) \\
\text{s.t.}      &a_i^{\top}Ha_i \leq n, ~~i = 1,\cdots,p.
\end{array}\right.
\end{equation}
The objective of this problem is to find the minimum ellipsoid that covers the points $a_1, \cdots, a_p \in \mathbb{R}^n$. 
The datasets $\set{a_i}_{i=1}^{p}$ are generated using independent multinomial Gaussian distribution $\Nc(0,\Sigma)$ as in \cite{damla2008linear}. 
For this problem, one state-of-the-art solver is the Frank-Wolfe algorithm with away-step \cite{lacoste2015global}. 
Its attraction is from the observation that the linesearch problem for determining optimal step-size $\tau$:
\begin{equation*}
\min_{\tau \in [0,1]} f((1 - \tau)x + \tau e_j)
\end{equation*}
has a closed-form solution (see \cite{Khachiyan1996} for more details). 
Therefore, we do not have to carry out a linesearch at each iteration of the Frank-Wolfe algorithm. 

Recently, \cite{damla2008linear} showed that the Frank-Wolfe away algorithm has a linear convergence rate for this specific problem. 
Figure \ref{fig:MVEE} reveals the performance of our algorithm (FWPN) and Frank-Wolfe algorithm with away-step on three datasets, where the dimension $n$ varies from $100$ to $5,000$. 
Note that existing literature only tested for problems with $n \leq 500$. 
As far as we are aware of, this is the first attempt to solve problem \eqref{eq:D_optimal_design_exam} with $n$ up to $n = 5,000$. 

Figure \ref{fig:MVEE} shows that when the size of the problem is small, our algorithm is slightly better than the Frank-Wolfe method with away-step. 
However, when the size of the problem becomes large, our algorithm highly outperforms the Frank-Wolfe method in terms of computational time. 
This happens due to a small number of projected Newton steps while each inner iteration requires significantly cheap computational time.

\beforesubsec
\subsection{Logistic Regression with Elastic-net}\label{subsec:logistic_regression_exam}
\aftersubsec

Finally, let us consider the following logistic regression with elastic-net regularizer:
\begin{equation}\label{eq:logistic_reg_exam}
%\begin{array}{ll}
\min_{x\in \R^p} \frac{1}{n}e^{\top}\log(e + \exp(A^{\top}x)) + \frac{\mu}{2}\norms{x}^2 + \rho\norms{x}_1,
%\end{array}
\end{equation}
where $e  := (1,1,\cdots, 1)^{\top}\in\R^n$, $A := [-y_1a_1,\cdots,-y_na_n] \in \R^{p\times n}$, and $(a_i, y_i) \in \R^p\times\set{-1,1}$ for $i = 1, \cdots ,n$.

It is well-known that \eqref{eq:logistic_reg_exam} is equivalent to the following problem with a suitable penalty parameter $\rho_1 > 0$:
\begin{equation}\label{eq:logistic_reg_exam1}
{\!\!\!\!}\begin{array}{ll}
\displaystyle\min_{x\in \R^p} &f(x) : = \frac{1}{n}e^{\top}\log(e + \exp(A^{\top}x)) + \frac{\mu}{2}\norms{x}^2\\
\text{s.t.}      &\norms{x}_1 \leq \rho_1.
\end{array}{\!\!\!\!}
\end{equation}
It is has been shown in \cite{SunTran2017gsc} that $f(x) := \frac{1}{n}e^{\top}\log(e + \exp(A^{\top}x)) + \frac{\mu}{2}\norms{x}^2$ is self-concordant.
Therefore, \eqref{eq:logistic_reg_exam1} fits into our template \eqref{eq:constr_cvx} with $\Xc := \set{x\in \R^p \mid \norms{x}_1 \leq \rho_1}$ in this case.

We compare Algorithm~\ref{alg:A1} (FWPN) with a standard proximal-gradient method \cite{Beck2009} and an accelerated proximal-gradient method with linesearch and restart \cite{Becker2011a,Su2014}.
These methods are abbreviated by PG and APG-LSRS, respectively.
We use binary classification datasets: \textbf{a1a}, \textbf{a9a}, \textbf{w1a}, \textbf{w8a}, \textbf{covtype}, \textbf{news20}, \textbf{real-sim} from \cite{CC01a} and generate the datasets \textbf{mnist17} and \textbf{mnist38} from the \textbf{mnist} dataset where digits are chosen from $\set{1,7}$ and $\set{3,8}$, respectively. 
We set $\mu := \frac{1}{n}$ as in \cite{SAGA}, and $\rho_1$ is set to be $10$, which guarantees that the sparsity of the solution is maintained between $1\%$ and $10\%$.

Since we need to evaluate the projection on a $\ell_1$-norm ball at each iteration of PG and APG-LSRS, we use the algorithm provided by \cite{Duchi2008a} which only need $\Oc(p)$ time. 
For our algorithm, since the $\ell_1$-norm ball is still a polytope, we can linearly solve the subproblem by using the Frank-Wolfe algorithm with away-step from \cite{lacoste2015global}. 
The performance and results of three algorithms on the above datasets are presented in Figure \ref{fig:logistic}.

From Figure \ref{fig:logistic}, one can observe that our algorithm outperforms PG and APG-LSRS both on small and large datasets. 
This happens thanks to the low computational cost of the linear oracle and the linear convergence of the FW method with away-step. 
It is interesting that although our algorithm is a hybrid method between second-order and first-order methods, we can still solve high-dimensional problems (e.g., when $p = 1,355,191$ in \textbf{news20} dataset) as often seen in first-order methods. 
We gain this efficiency due to the use of Hessian-vector products instead of full Hessian evaluations.

\beforesec
\section{Conclusion}\label{sec:conc}
\aftersec
\vspace{-0.75ex}
In this paper, we have combined the well-known Frank-Wolfe  (first-order) method and an inexact projected Newton (second-order) method to form a novel hybrid algorithm for solving a class of constrained convex problems under self-concordant structures. 
Our approach is different from existing methods that heavily rely on the Lipschitz continuous gradient assumption. 
Under this new setting, we give the first rigorous convergence and complexity analysis. 
Surprisingly, the LO complexity of our algorithm is still comparable with the Frank-Wolfe algorithms for a different class of problems. 
In addition, our algorithm enjoys several computation advantages on some specific problems, which are also supported by the three numerical examples in Section \ref{sec:experiment}. 
Moreover, the last example has shown that our algorithm still outperforms first-order methods on large-scale instances. 
%This seems a contradiction to some traditional viewpoints that this kind of large-scale problem instances can only efficiently be solved by advanced first-order methods. 
Our finding suggests that sometimes it is worth carefully combine first-order and second-order methods for solving large-scale problems in non-standard settings. 

\paragraph{Acknowledgments:}
Q. Tran-Dinh was partly supported by the National Science Foundation (NSF), grant No. 1619884 and the Office of Naval Research (ONR), grant No. N00014-20-1-2088.
V. Cevher was partly supported by  the European Research Council (ERC) under the European Union's Horizon 2020 research and innovation program (grant agreement n 725594 - time-data) and  by 2019 Google Faculty Research Award.

% \input{sections/refs_main}
%\newpage

\appendix
%\begin{center}
%\textsc{\large Supplementary document}
%
%\textbf{\Large A Frank-Wolfe-based Projected Newton Algorithm for Constrained \vspace{0.5ex}\\ Self-Concordant Minimization}
%\end{center}
%
%\vspace{1ex}
\beforesec
\section{Appendix: The proof of technical results}\label{sec:apdx_proofs}
\aftersec
Let us recall the following key properties of standard self-concordant functions.
Let $f$ be standard self-concordant and $x,y \in \dom{f}$ such that $\norms{y - x}_{x} < 1$ and $\norms{y - x}_{y} < 1$.
Then
\begin{equation}\label{eq:local_norm_prop1}
\begin{array}{l}
\left\{
\begin{array}{l}
\big(\norms{u}_{y}\big)^2 = u^{\top}\nabla^2f(y)u \leq u^{\top}\frac{\nabla^2f(x)}{(1-\norms{y-x}_{x})^2}u
= \left(\frac{\norms{u}_{x}}{1-\norms{y-x}_{x}}\right)^2,~~~\forall u\in \R^p,\\
\big(\norms{u}_{y}\big)^2 = u^{\top}\nabla^2f(y)u \leq u^{\top}\frac{\nabla^2f(x)}{(1-\norms{y-x}_{y})^2}u
= \left(\frac{\norms{u}_{x}}{1-\norms{y-x}_{y}}\right)^2,~~~\forall u\in \R^p.
\end{array}
\right.
\end{array}
\end{equation}
These inequalities can be found in \cite{Nesterov2004}[Theorem 4.1.6].

%% Proof of Lemma 3.1.
\beforesubsec
\subsection{The Proof of Lemma \ref{lm:inexact_sol_prop}}\label{apdx:lm:inexact_sol_prop}
\aftersubsec
%\begin{proof}
From Definition \ref{def:inexact_sol_sub_prob}, we have $\iprods{\nabla Q_f(T_{\eta}(x); x), T_{\eta}(x) - T(x)} \leq \eta^2$.
Since $Q_f(\cdot;x)$ is a convex quadratic function, it is easy to show that
\begin{equation*}
\begin{array}{ll}
&\iprods{\nabla Q_f(T_{\eta}(x);x), T_{\eta}(x) - T(x)} = \iprods{\nabla Q_f(T(x); x)  +  \nabla^2 f(x)(T_{\eta}(x) - T(x)), T_{\eta}(x) - T(x)} \leq \eta^2.
\end{array}
\end{equation*}
Substituting $T_{\eta}(x)$ for $u$ in the optimality condition \eqref{eq:opt_cond1}, we obtain $\iprods{\nabla{Q}_f(T(x); x), T_{\eta}(x) - T(x)} \geq 0$.
Combining the above two inequalities, we finally get
\begin{equation*}
\iprods{\nabla^2f(x)(T_{\eta}(x) - T(x)), T_{\eta}(x) - T(x)} \leq \eta^2,
\end{equation*}
which is equivalent to $\norms{T_{\eta}(x) - T(x)}_{x} \leq \eta$.
%\end{proof}
\Eproof
%% End of the proof.

\beforesubsec
\subsection{The Proof of Theorem \ref{tm:finite_step_phase1}}\label{apdx:proof_finite_step_phase1}
\aftersubsec
We would need two Lemmas to prove Theorem \ref{tm:finite_step_phase1}. The following lemma describles the decreasing of objective value when apply damped-step update.

\begin{lemma}\label{lm:one_iter_obj_value}
Let  $\gamma_k := \norms{z^k - x^k}_{x^k}$ be the local distance between $z^k$ to $x^k$, and  $\norms{z^k - T(x^k)}_{x^k} \leq \eta_k$. 
If we choose $\alpha \in (0, 1)$ such that  $\alpha\gamma_k < 1$ and update $x^{k+1} := x^k + \alpha(z^k - x^k)$, then we have
\begin{equation}\label{key_prop6_apdx}
f(x^{k+1}) \leq f(x^{k}) - \left[\alpha( \gamma_k^2 - \eta_k^2) - \omega_{*}(\alpha\gamma_k)\right].
\end{equation}
Assume $\gamma_k > \eta_k$. If $\delta \in (0,1)$ and the step size is $\alpha_k := \frac{\delta(\gamma_k^2 - \eta_k^2)}{\gamma_k(\gamma_k^2 + \gamma_k - \eta_k^2)}$ then we have $\alpha_k\gamma_k < \delta < 1$. 
Moreover, we have
\begin{equation}\label{key_prop7_apdx}
f(x^{k+1}) \leq f(x^{k}) - \delta\omega\left(\frac{\gamma_k^2 - \eta_k^2}{\gamma_k}\right),
\end{equation}
where $\omega(\tau) := \tau - \ln(1 + \tau)$ and $\omega_{\ast}(\tau) := -\tau - \ln(1 - \tau)$ are two nonnegative and convex functions.
\end{lemma}

\begin{proof}
From \eqref{alg:outer_iter_full_step} and the stop criterion of Algorithm \ref{alg:A2}, it is clear that $z^k$ satisfies
\begin{equation*}
\iprod{\nabla f(x^k) + \nabla^2 f(x^k)(z^k -x^k), z^k - x^k} \leq \eta_k^2,
\end{equation*}
which means that $z^k$ is an $\eta_k$-solution of \eqref{eq:sub_problem} at $x = x^k$. 
This inequality leads to
\begin{equation}\label{one_iter_obj_value_est1}
\iprod{\nabla f(x^k), z^k - x^k} \leq \eta_k^2 - \norms{z^k - x^k}_{x^k}^2.
\end{equation}
Therefore, using the self-concordance of $f$ \cite{Nesterov2004}[Theorem 4.1.8], we can derive
\begin{equation}\label{one_iter_obj_value_est2}
\begin{array}{ll}
f(x^{k+1}) &\leq f(x^k) + \iprod{\nabla f(x^k), x^{k+1} - x^k} + \omega_{*}(\norms{x^{k+1} - x^k}_{x^k}) \vspace{1ex}\\
& = f(x^k) + \alpha\iprod{\nabla f(x^k), z^k - x^k} + \omega_{*}(\alpha\norms{z^k - x^k}_{x^k}) \vspace{1ex}\\
& \overset{\eqref{one_iter_obj_value_est1}}{\leq} f(x^k) + \alpha\left(\eta_k^2 - \norms{z^k - x^k}_{x^k}^2\right) + \omega_{*}(\alpha\norms{z^k - x^k}_{x^k}) \vspace{1ex}\\
& = f(x^{k}) - [\alpha(\gamma_k^2 - \eta_k^2) - \omega_{*}(\alpha\gamma_k)].
\end{array}
\end{equation}
This is exactly \eqref{key_prop6_apdx}.

Assume that $\gamma_k^2 > \eta_k^2$. 
Define $\psi(\alpha) := \alpha(\gamma_k^2 - \eta_k^2) - \omega_{*}(\alpha\gamma_k)$ and plug $\alpha_k = \frac{\delta(\gamma_k^2 - \eta_k^2)}{\gamma_k(\gamma_k^2 + \gamma_k - \eta_k^2)}$ into $\psi(\alpha)$, we arrive at
\begin{equation}\label{one_iter_obj_value_est3}
\begin{array}{ll}
\psi(\alpha_k) &= \alpha_k(\gamma_k^2 - \eta_k^2) - \omega_{*}(\alpha_k\gamma_k) \vspace{1ex}\\
& = \alpha_k(\gamma_k^2 - \eta_k^2 + \gamma_k) + \ln(1 - \alpha_k\gamma_k) \vspace{1ex}\\
& = \frac{\delta(\gamma_k^2 - \eta_k^2)}{\gamma_k} + \ln(1 - \frac{\delta(\gamma_k^2 - \eta_k^2)}{\gamma_k^2 - \eta_k^2 + \gamma_k}) \vspace{1ex}\\
& \geq \frac{\delta(\gamma_k^2 - \eta_k^2)}{\gamma_k} + \delta\ln(1 - \frac{(\gamma_k^2 - \eta_k^2)}{\gamma_k^2 - \eta_k^2 + \gamma_k}) \vspace{1ex}\\
& = \delta\omega(\frac{\gamma_k^2 - \eta_k^2}{\gamma_k}),
\end{array}
\end{equation}
where we use $\ln(1 - \delta x) \geq \delta \ln(1 - x)$ in $x \in (0,1)$ for the inequality. Using \eqref{one_iter_obj_value_est2} and \eqref{one_iter_obj_value_est3} we proves \eqref{key_prop7_apdx}.
\end{proof}

The following lemma shows that the residual $\norms{x^{k} - x^{\star}}_{x^{\star}}$ can be bounded by the projected Newton decrement $\bar{\gamma}_k := \norms{T(x^k) - x^k}_{x^k}$.

\begin{lemma}\label{lm:bound_lambda_by_bar_lambda}
Let  $\lambda_k := \norms{x^{k} - x^{\star}}_{x^{\star}}$, $\bar{\gamma}_k := \norms{T(x^k) - x^k}_{x^k}$, $\gamma_k := \norms{z^k - x^k}_{x^k}$, and $h$ be defined by \eqref{eq:define_h}.
If $\gamma_k + \eta_k \in (0, C_2)$, then we have
\begin{equation}\label{key_prop8_apdx}
\lambda_k \leq h(\bar{\gamma}_k) \leq h(\gamma_k + \eta_k).
\end{equation}
\end{lemma}

%%% The proof of Lemma 3.3.
\begin{proof}
Firstly, we can write down the optimality condition of \eqref{eq:sub_problem} and \eqref{eq:constr_cvx}, respectively as follows:
\begin{equation*}
\left\{
\begin{array}{l}
  \iprod{\nabla f(x^k) + \nabla^2 f(x^k)[T(x^k) - x^k], x - T(x^k)} \geq 0,~~~\forall x \in \Xc, \vspace{1ex}\\
  \iprod{\nabla f(x^{\star}), x - x^{\star}} \geq 0,~~~\forall x \in \Xc.
\end{array}
\right.
\end{equation*}
Substituting $x^{\star}$ for $x$ into the first inequality and $T(x^k)$ for $x$ into the second inequality, respectively we get
\begin{equation*}
\left\{
\begin{array}{l}
  \iprod{\nabla f(x^k) + \nabla^2 f(x^k)[T(x^k) - x^k], x^{\star} - T(x^k)} \geq 0, \vspace{1ex}\\
  \iprod{\nabla f(x^{\star}), T(x^k) - x^{\star}} \geq 0.
\end{array}
\right.
\end{equation*}
Adding up both inequalities yields 
\begin{equation*}
\iprod{\nabla f(x^{\star}) - \nabla f(x^k) - \nabla^2 f(x^k)[T(x^k) - x^k], T(x^k) - x^{\star}} \geq 0,
\end{equation*}
which is equivalent to
\begin{equation*}
\iprod{\nabla f(T(x^k)) - \nabla f(x^k) - \nabla^2 f(x^k)[T(x^k) - x^k], T(x^k) - x^{\star}} \geq \iprod{\nabla f(T(x^k)) - \nabla f(x^{\star}) , T(x^k) - x^{\star}}.
\end{equation*}
Since $f$ is self-concordant, by \cite{Nesterov2004}[Theorem 4.1.7], we have 
\begin{equation*}
\iprod{\nabla f(T(x^k)) - \nabla f(x^{\star}) , T(x^k) - x^{\star}} \geq \frac{\norms{T(x^k) - x^{\star}}_{T(x^k)}^2}{1 + \norms{T(x^k) - x^{\star}}_{T(x^k)}}.
\end{equation*}
By the Cauchy-Schwarz inequality, this estimate leads to
\begin{equation}\label{bound_lambda_by_bar_lambda_est3}
\frac{\norms{T(x^k) - x^{\star}}_{T(x^k)}}{1 + \norms{T(x^k) - x^{\star}}_{T(x^k)}} \leq \norms{\nabla f(T(x^k)) - \nabla f(x^k) - \nabla^2 f(x^k)[T(x^k) - x^k]}_{T(x^k)}^*.
\end{equation}
Now, we can bound the right-hand side of the above inequality as
\begin{equation}\label{bound_lambda_by_bar_lambda_est4}
\begin{array}{ll}
\mathcal{R} &:= \norms{\nabla f(T(x^k)) - \nabla f(x^k) - \nabla^2 f(x^k)[T(x^k) - x^k]}_{T(x^k)}^* \vspace{1ex}\\
& \overset{\eqref{eq:local_norm_prop1}}{\leq} \frac{\norms{\nabla f(T(x^k)) - \nabla f(x^k) - \nabla^2 f(x^k)[T(x^k) - x^k]}_{x^k}^*}{1 - \norms{T(x^k) - x^k}_{x^k}} \vspace{1ex}\\
& \leq \left(\frac{\norms{T(x^k) - x^k}_{x^k}}{1 - \norms{T(x^k) - x^k}_{x^k}}\right)^2, 
\end{array}
\end{equation}
where the last inequality is from  \cite{TranDinh2016c}[Theorem 1].
From \eqref{bound_lambda_by_bar_lambda_est3} and \eqref{bound_lambda_by_bar_lambda_est4}, we have
\begin{equation*}
\frac{\norms{T(x^k) - x^{\star}}_{T(x^k)}}{1 + \norms{T(x^k) - x^{\star}}_{T(x^k)}} \leq \left(\frac{\norms{T(x^k) - x^k}_{x^k}}{1 - \norms{T(x^k) - x^k}_{x^k}}\right)^2,
\end{equation*}
which can be reformulated as
\begin{equation}\label{bound_lambda_by_bar_lambda_est1}
\norms{T(x^k) - x^{\star}}_{T(x^k)} \leq \frac{\norms{T(x^k) - x^k}_{x^k}^2}{1 - 2\norms{T(x^k) - x^k}_{x^k}}.
\end{equation}
Next, since we want to use $\norms{T(x^k) - x^k}_{x^k}$ to bound $\norms{x^k - x^{\star}}_{x^{k}}$, we can derive 
\begin{equation}\label{bound_lambda_by_bar_lambda_est2}
\begin{array}{ll}
\norms{x^k - x^{\star}}_{x^{k}} &\leq \norms{x^k - T(x^k)}_{x^k} + \norms{T(x^k) - x^{\star}}_{x^k} \vspace{1ex}\\
&\overset{\eqref{eq:local_norm_prop1}}{\leq} \norms{x^k - T(x^k)}_{x^k} + \frac{\norms{T(x^k) - x^{\star}}_{T(x^k)}}{1 - \norms{x^k - T(x^k)}_{x^k}} \vspace{1ex}\\
&\overset{\eqref{bound_lambda_by_bar_lambda_est1}}{\leq}  \norms{x^k - T(x^k)}_{x^k} + \frac{\norms{x^k - T(x^k)}_{x^k}^2}{(1 - 2\norms{x^k - T(x^k)}_{x^k})(1 - \norms{x^k - T(x^k)}_{x^k})} \vspace{1ex}\\
&= \bar{\gamma}_k + \frac{\bar{\gamma}_k^2}{(1 - 2\bar{\gamma}_k)(1 - \bar{\gamma}_k)}.
\end{array}
\end{equation}
Notice that $h$ is monotonically increasing and $\bar{\gamma_k} \leq \gamma_k + \eta_k$, we finally get
\begin{equation*}
\begin{array}{ll}
\norms{x^k - x^{\star}}_{x^{\star}} &\overset{\eqref{eq:local_norm_prop1}}{\leq} \frac{\norms{x^k - x^{\star}}_{x^k}}{1 - \norms{x^k - x^{\star}}_{x^k}} \vspace{1ex}\\
& \overset{\eqref{bound_lambda_by_bar_lambda_est2}}{\leq}  \frac{\bar{\gamma}_k(1 -2\bar{\gamma}_k + 2\bar{\gamma}_k^2)}{(1 - 2\bar{\gamma}_k)(1 - \bar{\gamma}_k)^2 - \bar{\gamma}_k^2} = h(\bar{\gamma}_k) \vspace{1ex}\\
& \leq h(\gamma_k + \eta_k),
\end{array}
\end{equation*}
which proves \eqref{key_prop8_apdx}.
\end{proof}

Now we can prove Theorem \ref{tm:finite_step_phase1}. We first restate the Theorem.
\begin{theorem}
Let $\omega(\tau) := \tau - \ln(1 + \tau)$. If we choose the parameters as in Algorithm \ref{alg:A1}, then after at most
\begin{equation}\label{eq:k_max_apdx}
K : = \frac{f(x^{0}) - f(x^{\star})}{\delta\omega\left(\frac{1 - 2C_1}{C_1}h^{-1}(\beta)\right)}.
\end{equation}
outer iterations of the damped-step scheme, we can guarantee that $\gamma_k + \eta_k \leq h^{-1}(\beta)$ for some $k \in [K]$, which implies that $\norms{x^{k} - x^{\star}}_{x^{\star}} \leq \beta$. 
Moreover, the total number of LO calls is at most 
\begin{equation*}
\Tc_1 := \frac{6D_{\Xc}^2\lambda_{\max}(\nabla^2 f(x^0))}{(C_1h^{-1}(\beta))^2}\frac{1 - (1-\delta)^{K+1}}{\delta(1-\delta)^K},
\end{equation*}
where $D_{\Xc} := \displaystyle\max_{x,y\in \Xc}\norms{x-y}$.
\end{theorem}

%%% The proof of Theorem 3.1.
\begin{proof}
Notice that in Algorithm \ref{alg:A1}, we always choose $\eta_k = C_1h^{-1}(\beta)$ in the damped-step stage. 
Clearly, if $\gamma_k + \eta_k > h^{-1}(\beta)$, then $\gamma_k > h^{-1}(\beta) - \eta_k = (1 -C_1)h^{-1}(\beta)$, where $C_1 \in  (0,0.5)$.
Therefore, we have
\begin{equation*}
\frac{\gamma_k^2 - \eta_k^2}{\gamma_k} \geq \frac{((1 -C_1)h^{-1}(\beta))^2 - (C_1h^{-1}(\beta))^2}{C_1h^{-1}(\beta)} = \frac{1 - 2C_1}{C_1}h^{-1}(\beta).
\end{equation*}
Using Lemma \ref{lm:one_iter_obj_value} and the monotonicity of $\omega$ we also have
\begin{equation*}
f(x^{k+1}) \overset{\eqref{key_prop7_apdx}}{\leq} f(x^{k}) - \delta\omega(\frac{\gamma_k^2 - \eta_k^2}{\gamma_k}) \leq f(x^k) - \delta\omega(\frac{1 - 2C_1}{C_1}h^{-1}(\beta)).
\end{equation*}
Consequently, we  need at most $K := \frac{f(x^{0}) - f(x^{\star})}{\delta\omega\left(\frac{1 - 2C_1}{C_1}h^{-1}(\beta)\right)}
$ outer iterations to get $\gamma_k + \eta_k \leq h^{-1}(\beta)$ as stated in \eqref{eq:k_max_apdx}.
 
From Lemma \ref{lm:bound_sub_solver}, we can show that the number of LO calls needed at the $k$-th outer iteration is $T_k := \frac{6\lambda_{\max}(\nabla^2f(x^k))D_{\Xc}^2}{\eta_k^2}$. 
Since $f$ is self-concordant, we have
\begin{equation*}
\nabla^2 f(x^{k+1}) \leq \frac{\nabla^2 f(x^{k})}{1 - \norms{x^{k+1} - x^{k}}_{x^k}} = \frac{\nabla^2 f(x^{k})}{1 - \alpha_k\gamma_k} \leq \frac{\nabla^2 f(x^{k})}{1 - \delta},
\end{equation*}
which implies that $\nabla^2 f(x^{k}) \leq \frac{\nabla^2 f(x^{0})}{(1 - \delta)^k}$.
Hence, the total number of LO calls can be computed by
\begin{equation*}
\Tc_1 := \sum_{k=0}^KT_k =  6D_{\Xc}^2\sum_{k = 0}^{K} \frac{\lambda_{\max}(\nabla^2 f(x^k))}{\eta_k^2} \leq \frac{6D_{\Xc}^2}{(C_1h^{-1}(\beta))^2}\sum_{k = 0}^{K} \frac{\lambda_{\max}(\nabla^2 f(x^0))}{(1-\delta)^k} = \frac{6D_{\Xc}^2\lambda_{\max}(\nabla^2 f(x^0))}{(C_1h^{-1}(\beta))^2}\frac{1 - (1-\delta)^{K+1}}{\delta(1-\delta)^K}.
\end{equation*}
Finally, if $\gamma_k + \eta_k \leq h^{-1}(\beta)$, then we have $\lambda_k \overset{\eqref{key_prop8_apdx}}{\leq} h(\gamma_k + \eta_k) \leq h(h^{-1}(\beta)) = \beta$.
\end{proof}

\beforesubsec
\subsection{The Proof of Theorem \ref{tm:linear_convergence}}\label{apdx:proof_linear_convergence}
\aftersubsec

The  following lemma shows that $\norms{x^{k+1} - x^{\star}}_{x^{\star}}$ and $\norms{x^{k+1} - x^k}_{x^k}$ can both be bounded by $\norms{x^{k} - x^{\star}}_{x^{\star}}$ when $\norms{x^{k} - x^{\star}}_{x^{\star}}$ is sufficiently small. 

\begin{lemma}\label{lm:one_iter_ineq}
Suppose that $\lambda_k := \norms{x^k - x^{\star}}_{x^{\star}} \leq \beta$, where $\beta \in (0, 1)$ is chosen by Algorithm \ref{alg:A1}. 
Then, we have
\begin{equation}\label{key_prop1_apdx}
\lambda_{k+1} \leq \frac{\eta_k}{1-\lambda_k} + \frac{\lambda_k^2}{(1 - \lambda_k)^2(1-2\lambda_k)}.
\end{equation}
In addition, we can also bound $\norms{x^{k+1} - x^k}_{x^k}$ as follows:
\begin{equation}\label{key_prop2_apdx}
\norms{x^{k+1} - x^k}_{x^k} \leq \eta_k + \frac{\lambda_k^2}{(1-2\lambda_k)(1-\lambda_k)} + \frac{\lambda_k}{1 - \lambda_k}.
\end{equation}
\end{lemma}

\begin{proof}
Since we always choose full-step $\alpha_k = 1$, we have $x^{k+1} = z^{k+1}$. 
Therefore, $\norms{x^{k+1} - T(x^k)}_{x^k} = \norms{z^k - T(x^k)}_{x^k} \leq \eta_k$, which leads to
\begin{equation}\label{eq:one_iter_ineq_est1}
\begin{array}{ll}
\lambda_{k+1} &= \norms{x^{k+1} - x^{\star}}_{x^{\star}} \leq \norms{x^{k+1} - T(x^k)}_{x^{\star}} +  \norms{T(x^k) - x^{\star}}_{x^{\star}} \vspace{1ex}\\
&\overset{\eqref{eq:local_norm_prop1}}{\leq} \frac{\norms{x^{k+1} - T(x^k)}_{x^k}}{1 - \norms{x^k - x^{\star}}_{x^{\star}}} + \frac{\norms{T(x^k) - x^{\star}}_{x^k}}{1 - \norms{x^k - x^{\star}}_{x^{\star}}} \vspace{1ex}\\
& \leq \frac{\eta_k}{1 - \lambda_k} + \frac{\norms{T(x^k) - x^{\star}}_{x^k}}{1 - \lambda_k}.
\end{array}
\end{equation}
This proves \eqref{key_prop1_apdx}.

Now, we bound $\norms{T(x^k) - x^{\star}}_{x^k}$ as follows. 
Firstly, the optimality conditions of \eqref{eq:sub_problem} and \eqref{eq:constr_cvx} can be written as
\begin{equation*}
\left\{
\begin{array}{ll}
\iprods{\nabla f(x^k) + \nabla^2f(x^k)(T(x^k) - x^k), x - T(x^k)} \geq 0, ~~\forall x \in \Xc, \vspace{1ex}\\
\iprods{\nabla f(x^{\star}), x - x^{\star}} \geq 0,~~\forall x\in\Xc. \vspace{1ex}\\
\end{array}
\right.
\end{equation*}
This can be rewritten equivalently to
\begin{equation}\label{eq:one_iter_ineq_est2}
\left\{
\begin{array}{ll}
\iprods{\nabla^2f(x^k)[T(x^k) - (x^k - \nabla^2f(x^k)^{-1}\nabla f(x^k))], x - T(x^k)} \geq 0, ~~\forall x \in \Xc, \vspace{1ex}\\
\iprods{\nabla^2f(x^k)[x^{\star} - (x^{\star} - \nabla^2f(x^k)^{-1}\nabla f(x^{\star}))], x - x^{\star}} \geq 0,~~\forall x\in\Xc. \vspace{1ex}\\
\end{array}
\right.
\end{equation}
Similar to the proof of \cite{Bauschke2011}[Theorem 3.14], we can show that \eqref{eq:one_iter_ineq_est2} is equivalent to
\begin{equation}\label{eq:one_iter_ineq_est3}
\left\{
\begin{array}{rl}
  T(x^k) &= \proj_{\Xc}^{\nabla^2f(x^k)}\left(x^k - \nabla^2f(x^k)^{-1}\nabla f(x^k)\right), \vspace{1ex}\\
  x^{\star} &= \proj_{\Xc}^{\nabla^2f(x^k)}\left(x^{\star} - \nabla^2f(x^k)^{-1}\nabla f(x^{\star})\right). \vspace{1ex}\\
\end{array}
\right.
\end{equation}
Using the nonexpansiveness of the projection operator we can derive
\begin{equation}\label{eq:one_iter_ineq_est4}
\begin{array}{ll}
\norms{T(x^k) - x^{\star}}_{x^k} &\overset{\eqref{eq:one_iter_ineq_est3}}{=} \Big\Vert \proj_{\Xc}^{\nabla^2f(x^k)}\left(x^k - \nabla^2f(x^k)^{-1}\nabla f(x^k)\right)  -  \proj_{\Xc}^{\nabla^2f(x^k)}\left(x^{\star} - \nabla^2f(x^k)^{-1}\nabla f(x^{\star})\right)\Big\Vert_{x^k} \vspace{1ex}\\
&\leq \norms{x^k - x^{\star} - \nabla^2f(x^k)^{-1}(\nabla f(x^k) - \nabla f(x^{\star}))}_{x^k} \vspace{1ex}\\
& = \norms{\nabla f(x^{\star}) - \nabla f(x^k) - \nabla^2f(x^k)(x^{\star} - x^k)}_{x^k}^{\ast} \vspace{1ex}\\
& \leq \frac{\norms{x^{\star} - x^k}_{x^k}^2}{1 - \norms{x^{\star} - x^k}_{x^k}} \vspace{1ex}\\
& \overset{\eqref{eq:local_norm_prop1}}{\leq} \frac{\norms{x^{\star} - x^k}_{x^{\star}}^2}{(1 - 2\norms{x^{\star} - x^k}_{x^{\star}})(1 - \norms{x^{\star} - x^k}_{x^{\star}})} \vspace{1ex}\\
& = \frac{\lambda_k^2}{(1-2\lambda_k)(1-\lambda_k)},
\end{array}
\end{equation}
where the second last inequality is from \cite{TranDinh2016c}[Theorem 1].
Plugging \eqref{eq:one_iter_ineq_est4} into \eqref{eq:one_iter_ineq_est1}, we get \eqref{key_prop1_apdx}.

Finally, we note that
\begin{equation*}
\begin{array}{ll}
\norms{x^{k+1} - x^k}_{x^k} &\leq \norms{x^{k+1} - T(x^k)}_{x^k} + \norms{x^{\star} - T(x^k)}_{x^k} + \norms{x^k - x^{\star}}_{x^k}\vspace{1ex}\\
&\overset{\eqref{eq:one_iter_ineq_est4}}\leq \norms{x^{k+1} - T(x^k)}_{x^k} + \frac{\lambda_k^2}{(1-2\lambda_k)(1-\lambda_k)} + \norms{x^k - x^{\star}}_{x^k}\vspace{1ex}\\
&\overset{\eqref{eq:local_norm_prop1}}{\leq} \eta_k + \frac{\lambda_k^2}{(1-2\lambda_k)(1-\lambda_k)} + \frac{\norms{x^k - x^{\star}}_{x^{\star}}}{1 - \norms{x^k - x^{\star}}_{x^{\star}}} \vspace{1ex}\\
&= \eta_k + \frac{\lambda_k^2}{(1-2\lambda_k)(1-\lambda_k)} + \frac{\lambda_k}{1 - \lambda_k},
\end{array}
\end{equation*}
which proves \eqref{key_prop2_apdx}.
\end{proof}

Now we can prove Theorem \ref{tm:linear_convergence}. We first restate the Theorem.
\begin{theorem}
Suppose $\norms{x^0 - x^{\star}}_{x^{\star}} \leq \beta$ and the triple $(\sigma, \beta, C)$ satisfies the following condition:
\begin{equation}\label{eq:condition_para_apdx}
\left\{\begin{array}{l}
  \sigma \in (0,1), ~~\beta \in (0, 0.5), ~~C > 1 \vspace{1ex}\\
  \frac{1}{C(1-\beta)} + \frac{\beta}{(1-2\beta)(1-\beta)^2}\leq \sigma, \vspace{1ex}\\
  \frac{1}{C} + \frac{1}{(1-2\beta)}\leq 2.
\end{array}\right.
\end{equation}
In addition, we set $\eta_k := \frac{\beta\sigma^k}{C}$ and update $\{x^k\}$ by full-step in Algorithm \ref{alg:A1}. 
Then, for $k\geq 0$, we have 
\begin{equation}\label{key_prop3_apdx}
\norms{x^k - x^{\star}}_{x^{\star}} \leq \beta\sigma^k ~~~~\text{and}~~~~\norms{x^{k+1} - x^{k}}_{x^{k}} \leq 2\beta\sigma^k.
\end{equation}
\end{theorem}

\begin{proof}
We prove this theorem by induction.
Firstly, we have $\norms{x^0 - x^{\star}}_{x^{\star}} \leq \beta\sigma^0 = \beta  < 1$ by assumption. 
Next, suppose that $\lambda_k := \norms{x^k - x^{\star}}_{x^{\star}} \leq \beta\sigma^k$ for $k \geq 0$.
We can derive
\begin{equation*}
\begin{array}{ll}
\lambda_{k+1} &:= \norms{x^{k+1} - x^{\star}}_{x^{\star}} \vspace{1ex}\\
& \overset{\eqref{key_prop1_apdx}}{\leq} \frac{\eta_k}{1-\lambda_k} + \frac{\lambda_k^2}{(1 - \lambda_k)^2(1-2\lambda_k)} \vspace{1ex}\\
&= \frac{\beta\sigma^k}{C(1-\lambda_k)} + \frac{\lambda_k^2}{(1 - \lambda_k)^2(1-2\lambda_k)} \vspace{1ex}\\
&\leq \left(\frac{1}{C(1-\lambda_k)} + \frac{\lambda_k}{(1 - \lambda_k)^2(1-2\lambda_k)}\right)\beta\sigma^k ~~(\text{by induction})\vspace{1ex}\\
&\leq \left(\frac{1}{C(1-\beta)} + \frac{\beta}{(1 - \beta)^2(1-2\beta)}\right)\beta\sigma^k ~~(\text{by induction})\vspace{1ex}\\
&\overset{\eqref{eq:condition_para_apdx}}{\leq} \beta\sigma^{k+1},
\end{array}
\end{equation*}
which proves the first estimate of \eqref{key_prop3_apdx}.

Similarly, we also have
\begin{equation*}
\begin{array}{ll}
\norms{x^{k+1} - x^k}_{x^k} &\overset{\eqref{key_prop2_apdx}}{\leq} \eta_k + \frac{\lambda_k^2}{(1-2\lambda_k)(1-\lambda_k)} + \frac{\lambda_k}{1 - \lambda_k} \vspace{1ex}\\
&= \frac{\beta\sigma^k}{C} + \frac{\lambda_k^2}{(1-2\lambda_k)(1-\lambda_k)} + \frac{\lambda_k}{1 - \lambda_k} \vspace{1ex}\\
&\leq \left(\frac{1}{C} + \frac{\lambda_k}{(1 - \lambda_k)(1-2\lambda_k)} + \frac{1}{1 - \lambda_k} \right)\beta\sigma^k ~~(\text{by induction})\vspace{1ex}\\
&\leq \left(\frac{1}{C} + \frac{\beta}{(1 - \beta)(1-2\beta)} + \frac{1}{1 - \beta} \right)\beta\sigma^k ~~(\text{by induction})\vspace{1ex}\\
&= \left(\frac{1}{C} + \frac{1}{(1 - 2\beta)} \right)\beta\sigma^k \vspace{1ex}\\
& \overset{\eqref{eq:condition_para_apdx}}{\leq} 2\beta\sigma^{k},
\end{array}
\end{equation*}
which proves the second estimate of \eqref{key_prop3_apdx}.
\end{proof}

\beforesubsec
\subsection{The proof of Theorem \ref{tm:complex_analysis}}\label{apdx:proof_complex_analysis}
\aftersubsec

Firstly, the following lemma establishes the sublinear convergence rate of the Frank-Wolfe gap in each outer iteration.
%%% Lemma 3.5.
\begin{lemma}\label{lm:bound_sub_solver}
At the $k$-th outer iteration of Algorithm~\ref{alg:A1}, if we run the Frank-Wolfe subroutine  \eqref{alg:outer_iter_full_step} to update $u^t$, then, after $T_k$ iterations, we have
\begin{equation}\label{key_prop4_apdx}
\min_{t = 1,\cdots, T_k} V_k(u^t) \leq \frac{6\lambda_{\max}(\nabla^2f(x^k))D_{\Xc}^2}{T_k + 1},
\end{equation}
where $V_k(u^t) := \max_{u\in\Xc}\iprod{\nabla f(x^k) + \nabla^2f(x^k)(u^t-x^k), u^t - u}$. 
As a result, the number of LO calls at the $k$-th outer iteration of Algorithm \ref{alg:A1} is at most $O_k := \frac{6\lambda_{\max}(\nabla^2f(x^k))D_{\Xc}^2}{\eta_k^2}$.
\end{lemma}

%%% The proof of Lemma 3.5.
\begin{proof}
Let $\phi_k(u) = \iprod{\nabla f(x^k), u - x^k} + 1/2\iprod{\nabla^2 f(x^k)(u - x^k), u - x^k}$ and $\set{u^t}$ be the sequence generated by the Frank-Wolfe subroutine \eqref{alg:outer_iter_full_step}. 
It is well-known that (see \cite{Jaggi2013}[Theorem 1])
\begin{equation}\label{eq:bound_sub_solver_est1}
\phi_k(u^t) - \phi_k^{\star} \leq \frac{2\lambda_{\max}(\nabla^2 f(x^k))D_{\Xc}^2}{t + 1}.
\end{equation}
Let $v^t := \arg\min_{u \in  \Xc}\{\iprod{\nabla \phi_k(u^t), u}\}$. 
Notice that
\begin{equation*}
\begin{array}{ll}
\phi_k(u^{t+1}) &= \min_{\tau \in [0,1]}\{\phi_k((1-\tau)u^t + \tau v^t)\} \leq \phi_k\left(\left(1 - \frac{2}{t+1}\right)u^t + \frac{2}{t+1}v^t\right) \vspace{1ex}\\
&\leq \phi_k(u^t) + \frac{2}{t+1}\iprod{\nabla \phi_k(u^t), (v^t - u^t)} + \frac{\lambda_{\max}(\nabla^2f(x^k))}{2}\left(\frac{2}{t+1}\right)^2\norms{v^t - u^t}^2 \vspace{1ex}\\
&\leq \phi_k(u^t) - \frac{2}{t+1}V_k(u^t) + \frac{2\lambda_{\max}(\nabla^2f(x^k))}{(t+1)^2}D_{\Xc}^2.  
\end{array}
\end{equation*}
This is equivalent to
\begin{equation}\label{eq:bound_sub_solver_est2}
tV_k(u^t) \leq \frac{t(t+1)}{2}\left(\phi_k(u^t) - \phi_k(u^{t+1})\right) + \frac{t\lambda_{\max}(\nabla^2f(x^k))}{t+1}D_{\Xc}^2.
\end{equation}
Summing  up this inequality from $t = 1$ to $T_k$, we get
\begin{equation*}
\begin{array}{ll}
\displaystyle\frac{T_k(T_k+1)}{2}\min_{t = 1,\cdots, T_k}\set{V_k(u^t)} &\leq \displaystyle\sum_{t = 1}^{T_k} tV_k(u^t) \vspace{1ex}\\
&\overset{\eqref{eq:bound_sub_solver_est2}}{\leq} \displaystyle\sum_{t = 1}^{T_k} t\phi_k(u^t) - \frac{T_k(T_k + 1)}{2}\phi_k(u^{T_k}) + T_k \lambda_{\max}(\nabla^2f(x^k))D_{\Xc}^2 \vspace{1ex}\\
&\leq \displaystyle\sum_{t = 1}^{T_k} t(\phi_k(u^t) - \phi_k^{\star}) + T_k \lambda_{\max}(\nabla^2f(x^k))D_{\Xc}^2 \vspace{1ex}\\
& \overset{\eqref{eq:bound_sub_solver_est1}}{\leq} 3T_k \lambda_{\max}(\nabla^2f(x^k))D_{\Xc}^2,
\end{array}
\end{equation*}
which implies \eqref{key_prop4_apdx}.
\end{proof}

Now we can prove Theorem \ref{tm:complex_analysis}. We first restate the Theorem.
\begin{theorem}
Suppose that $\norms{x^{0} - x^{\star}}_{x^{\star}} \leq \beta$. 
If we choose the parameters $\beta$, $\sigma$, $C$, and $\{\eta_k\}$ as in Theorem \ref{tm:linear_convergence} and update $\set{x^k}$ by the full-steps, then to obtain an $\varepsilon$-solution $x_{\varepsilon}^{\star}$ defined by \eqref{def:inexact_sol_prob}, it requires 
\begin{equation*}
\left\{\begin{array}{ll}
\left\lfloor\frac{\ln(\varepsilon)}{\ln(\sigma)}\right\rfloor + 1 = \BigO{\ln(\varepsilon^{-1})}~~\text{gradient evaluations $\nabla{f}(x^k)$}, \vspace{1ex}\\
\left\lfloor\frac{\ln(\varepsilon)}{\ln(\sigma)}\right\rfloor + 1 = \BigO{\ln(\varepsilon^{-1})}~~\text{Hessian evaluations $\nabla^2{f}(x^k)$, and} \vspace{1ex}\\
\left\lfloor \frac{3C^2\lambda_{\max}(\nabla^2{f}(x^0)) D_{\Xc}^2}{(1-2\beta)\beta^3\sigma^2} \cdot \varepsilon^{-2\nu}\right\rfloor = \BigO{\varepsilon^{-2\nu}} ~~\text{LO calls, with $\nu := 1 + \frac{\ln(1-2\beta)}{\ln(\sigma)}$}.
\end{array}\right.
\end{equation*}
\end{theorem}

%%% Proof of Theorem 3.3.
\begin{proof}
By self-concordance of $f$, using \cite{Nesterov2004}[Theorem 4.1.6], it holds that
\begin{equation*}
\nabla^2 f(x^{k+1}) \preceq \frac{1}{(1 - \norms{x^{k+1} - x^k}_{x^k})^2} \nabla^2 f(x^k) \overset{\eqref{key_prop3_apdx}}{\preceq} \frac{1}{(1 - 2\beta\sigma^k)^2} \nabla^2 f(x^k) \preceq \frac{1}{(1 - 2\beta)^2} \nabla^2 f(x^k).
\end{equation*}
By induction, we have
\begin{equation*}
\nabla^2 f(x^{k}) \preceq \left(\frac{1}{1 - 2\beta}\right)^{2k} \nabla^2 f(x^0).
\end{equation*}
Therefore, we can bound the maximum eigenvalue of $\nabla^2{f}(x^k)$ as
\begin{equation}\label{complex_analysis_est1}
\lambda_{\max}(\nabla^2 f(x^{k})) \leq \left(\frac{1}{1 - 2\beta}\right)^{2k} \lambda_{\max}(\nabla^2 f(x^0)).
\end{equation}
Let us denote by $\bar{\lambda}_0 := \lambda_{\max}(\nabla^2 f(x^0))$.
Then, fromLemma \ref{lm:bound_sub_solver}, we can see that the number of LO calls at the $k$-th outer iteration is at most
\begin{equation}\label{complex_analysis_est2}
\mathcal{O}_k := \frac{6\lambda_{\max}(\nabla^2 f(x^k))D_{\Xc}^2}{\eta_k^2} \overset{\eqref{complex_analysis_est1}}{\leq} \frac{6\bar{\lambda}_0D_{\Xc}^2}{(1 - 2\beta)^{2k}\eta_k^2} = \frac{6C^2\bar{\lambda}_0D_{\Xc}^2}{\beta^2((1 - 2\beta)\sigma)^{2k}},
\end{equation}
where the last equality holds because we set $\eta_k := \beta\sigma^k/C$ in Theorem \ref{tm:linear_convergence}.

To obtain an $\varepsilon$-solution $x^k$ defined by \eqref{def:inexact_sol_prob}, we need to impose $\beta\sigma^k \leq \varepsilon$ (recall that $\norms{x^k - x^{\star}}_{x^{\star}} \leq \beta\sigma^k$ by Theorem \ref{tm:linear_convergence}), which is equivalent to $k \geq \frac{\ln(\beta/\varepsilon)}{\ln(1/\sigma)}$. 
Since $\beta \in (0, 1)$, the outer iteration number is at most $\frac{\ln(1/\varepsilon)}{\ln(1/\sigma)} = \frac{\ln(\varepsilon)}{\ln(\sigma)}$. 
This number is also the total number of gradient and Hessian evaluations.

Finally, using \eqref{complex_analysis_est2}, the total number of LO calls in the entire algorithm can be estimated as
\begin{equation*}
\begin{array}{lll}
\Tc_2 &:= \sum_{k = 0}^{\frac{\ln(\varepsilon)}{\ln(\sigma)}} \frac{6C^2\bar{\lambda}_0 D_{\Xc}^2}{\beta^2((1 - 2\beta)\sigma)^{2k}}  &= \frac{6C^2\bar{\lambda}_0 D_{\Xc}^2}{\beta^2} \sum_{k = 0}^{\frac{\ln(\varepsilon)}{\ln(\sigma)}} \left(\frac{1}{(1 - 2\beta)\sigma}\right)^{2k} \vspace{1ex}\\
&\leq  \frac{3C^2\bar{\lambda}_0 D_{\Xc}^2}{(1-2\beta)\beta^3\sigma^2}\left(\frac{1}{(1 - 2\beta)\sigma}\right)^{\frac{2\ln(\varepsilon)}{\ln(\sigma)}}  &= \frac{3C^2\bar{\lambda}_0 D_{\Xc}^2}{(1-2\beta)\beta^3\sigma^2}\left(\frac{1}{\varepsilon}\right)^{2\left(1 + \frac{\ln(1-2\beta)}{\ln(\sigma)}\right)},
\end{array}
\end{equation*}
where the last equality holds since $\tau^{\alpha\ln(s)} = s^{\alpha\ln(\tau)}$.
\end{proof}

\beforesubsec
\subsection{The Proof of  Theorem \ref{tm:complex_analysis_obj_value}}\label{apdx:proof_complex_analysis_obj_value}
\aftersubsec
The following lemma states that we can bound $f(x^k) - f^{\star}$ by $\norms{x^{k+1} - x^k}_{x^k}$ and $\norms{x^{k} - x^{\star}}_{x^{\star}}$. 
Therefore, from the convergence rate of $\norms{x^{k+1} - x^k}_{x^k}$ and $\norms{x^{k} - x^{\star}}_{x^{\star}}$ in Theorem \ref{tm:linear_convergence}, we can obtain a convergence rate of $\set{f(x^k) - f^{\star}}$.

\begin{lemma}\label{lm:bound_obj_value}
Let $\gamma_k := \norms{x^{k+1} - x^k}_{x^k} = \norms{z^{k} - x^k}_{x^k}$ and  $\lambda_k := \norms{x^{k} - x^{\star}}_{x^{\star}}$. 
Suppose that $x^0\in\dom{f}\cap\Xc$. 
If $0 < \gamma_k,\lambda_k, \lambda_{k+1} < 1$, then we have
\begin{equation}\label{key_prop5_apdx}
f(x^{k+1}) \leq f(x^{\star}) + \frac{\gamma_k^2(\gamma_k + \lambda_k)}{1 - \gamma_k} + \eta_k^2 + \omega_{\ast}(\lambda_{k+1}),	
\end{equation}
where $\omega_{\ast}(\tau) := -\tau - \ln(1-\tau)$.
\end{lemma}

\begin{proof}
Firstly, from \cite{Nesterov2004}[Theorem 4.1.8], we have
\begin{equation*}
f(x^{k+1}) \leq f(x^{\star}) + \iprod{\nabla f(x^{\star}), x^{k+1} - x^{\star}} + \omega_{*}(\norms{x^{k+1} - x^{\star}}_{x^{\star}}),
\end{equation*}
provided that $\norms{x^{k+1} - x^{\star}}_{x^{\star}} < 1$.
Next, using $\iprod{\nabla f(x^{\star}) - \nabla f(x^{k+1}), x^{k+1} - x^{\star}} \leq 0$, we can further derive
\begin{equation}\label{bound_obj_value_est1}
f(x^{k+1}) \leq f(x^{\star}) + \iprod{\nabla f(x^{k+1}), x^{k+1} - x^{\star}} + \omega_{*}(\norms{x^{k+1} - x^{\star}}_{x^{\star}}).
\end{equation}
Now, we bound $\iprod{\nabla f(x^{k+1}), x^{k+1} - x^{\star}}$ as follows. 
Notice that this term can be decomposed as
\begin{equation*}
\begin{array}{ll}
\iprod{\nabla f(x^{k+1}), x^{k+1} - x^{\star}} &= \underbrace{\iprod{\nabla f(x^k) + \nabla^2 f(x^k)(x^{k+1} - x^k), x^{k+1} - x^{\star}}}_{\Tc_1}\vspace{2ex}\\
&+ \underbrace{\iprod{\nabla f(x^{k+1}) - \nabla f(x^k) - \nabla^2 f(x^k)(x^{k+1} - x^k), x^{k+1} - x^{\star}}}_{\Tc_2}. \vspace{1ex}\\
\end{array}
\end{equation*}
Since $x^{k+1}$ is an $\eta^k$-solution of \eqref{eq:sub_problem} at $x = x^k$, we have	
\begin{equation}\label{bound_obj_value_est2}
\Tc_1 = \iprod{\nabla f(x^k) + \nabla^2 f(x^k)(x^{k+1} - x^k), x^{k+1} - x^{\star}} \leq \eta_k^2.
\end{equation}
Using the Cauchy-Schwarz inequality and the triangle inequality, $\Tc_2$ can also be bounded as  
\begin{equation}\label{bound_obj_value_est3}
\begin{array}{ll}
\Tc_2 &= \iprod{\nabla f(x^{k+1}) - \nabla f(x^k) - \nabla^2 f(x^k)(x^{k+1} - x^k), x^{k+1} - x^{\star}} \vspace{1ex}\\
&\leq \norms{\nabla f(x^{k+1}) - \nabla f(x^k) - \nabla^2 f(x^k)(x^{k+1} - x^k)}_{x^k}^*\norms{x^{k+1} - x^{\star}}_{x^k} \vspace{1ex}\\
&\leq \frac{\norms{x^{k+1} - x^k}_{x^k}^2}{1 - \norms{x^{k+1} - x^k}_{x^k}}\norms{x^{k+1} - x^{\star}}_{x^k} \vspace{1ex}\\
&\leq \frac{\norms{x^{k+1} - x^k}_{x^k}^2}{1 - \norms{x^{k+1} - x^k}_{x^k}}\left[ \norms{x^{k} - x^{\star}}_{x^k} + \norms{x^{k+1} - x^{k}}_{x^k}\right] \vspace{1ex}\\
& = \frac{\gamma_k^2(\gamma_k + \lambda_k)}{1 - \gamma_k},
\end{array}
\end{equation}
where, the second inequality is due to \cite{TranDinh2016c}[Theorem 1]. 
Finally, we can bound $f(x^{k+1}) - f^{\star}$ as
\begin{equation*} 
\begin{array}{rl}
f(x^{k+1}) - f(x^{\star}) & \overset{\eqref{bound_obj_value_est1}}{\leq} \iprod{\nabla f(x^{k+1}), x^{k+1} - x^{\star}} + \omega_{*}(\lambda_{k+1}) \vspace{1ex}\\
& = \Tc_1 + \Tc_2 + \omega_{*}(\lambda_{k+1}) \vspace{1ex}\\
& \overset{\eqref{bound_obj_value_est2}\eqref{bound_obj_value_est3}}{\leq} \eta_k^2 + \frac{\gamma_k^2(\gamma_k + \lambda_k)}{1 - \gamma_k} + \omega_{*}(\lambda_{k+1}),
\end{array}
\end{equation*}
which proves \eqref{key_prop5_apdx}.
\end{proof}

Now we can prove Theorem \ref{tm:complex_analysis_obj_value}. We first restate the Theorem.
%%% Theorem 3.4.
\begin{theorem}
Suppose that $\norms{x^{0} - x^{\star}}_{x^{\star}} \leq \beta$. 
If we choose $\sigma$, $\beta$, $C$, and $\{\eta_k\}$ as in Theorem \ref{tm:linear_convergence} and update $\{x^k\}$ by the full-steps, then we have
\begin{equation}
f(x^{k+1}) - f(x^{\star}) \leq \left(\frac{12\beta^3}{1-2\beta} + \frac{\beta^2}{C^2} + \beta^2\right)\sigma^{2k}.
\end{equation}
Consequently, the total LO complexity of Algorithm~\ref{alg:A1} to achieve an $\varepsilon$-solution $x_{\varepsilon}^{\star}$ such that $f(x_{\varepsilon}^{\star}) - f^{\star} \leq \varepsilon$ is $\BigO{\varepsilon^{-\nu}}$, where $\nu := 1 + \frac{\ln(1-2\beta)}{\ln(\sigma)}$.
\end{theorem}

%%% The proof of Theorem 3.4.
\begin{proof}
It is easy to check that $\omega_{*}(\tau) \leq \tau^2$ for $0 < \tau < 0.5$. 
Therefore, there exists $k_0 > 0$ such that $\omega_{*}(\beta \sigma^k) \leq (\beta \sigma^k)^2$ for $k \geq k_0$. 
Since $\eta_k := \frac{\beta\sigma^k}{C}$, $\gamma_k \leq 2\beta\sigma^k$, and $\lambda_k \leq \beta\sigma^k$ in Theorem \ref{tm:linear_convergence}, for $k\geq k_0$, we have
\begin{equation}
\begin{array}{ll}
f(x^{k+1}) - f(x^{\star}) &\overset{\eqref{key_prop5_apdx}}{\leq} \frac{\gamma_k^2(\gamma_k + \lambda_k)}{1 - \gamma_k} + \eta_k^2 + \omega_{*}(\lambda_{k+1}) \vspace{1ex}\\
& \overset{\eqref{key_prop3_apdx}}{\leq} \frac{12\beta^3\sigma^{3k}}{1-2\beta\sigma^k} + \frac{\beta^2\sigma^{2k}}{C^2} + \omega_{*}(\beta\sigma^{k+1}) \vspace{1ex}\\
&\leq \left(\frac{12\beta^3\sigma^{k}}{1-2\beta\sigma^k} + \frac{\beta^2}{C^2} +\beta^2\sigma^2\right)\sigma^{2k} \vspace{1ex}\\
&\leq \left(\frac{12\beta^3}{1-2\beta} + \frac{\beta^2}{C^2} +\beta^2\right)\sigma^{2k}.
\end{array}
\end{equation}
Let $C_1 > \frac{12\beta^2}{1-2\beta} + \frac{\beta^2}{C^2} +\beta^2$ be a constant. 
To guarantee $f(x^{k+1}) - f(x^{\star}) \leq \varepsilon$, we impose $C_1\sigma^{2k} \leq \varepsilon$ i.e. $k \geq \frac{\ln(\varepsilon/C_1)}{2\ln(\sigma)}$. 
Therefore, the outer iteration number is at most $\frac{\ln(\varepsilon/C_1)}{2\ln(\sigma)} + 1$. 
Using \eqref{complex_analysis_est2}, the total number of LO calls will be
\begin{equation*}
\begin{array}{ll}
\sum_{k = 0}^{\frac{\ln(\varepsilon/C_1)}{2\ln(\sigma)} + 1} \frac{6C^2\lambda_{\max}(\nabla^2 f(x^0))D_{\Xc}^2}{\beta^2((1 - 2\beta)\sigma)^{2k}} &= \BigO{\sum_{k = 0}^{\frac{\ln(\varepsilon/C_1)}{2\ln(\sigma)} + 1} \left(\frac{1}{(1 - 2\beta)\sigma}\right)^{2k}} \vspace{1ex}\\
= \BigO{\left(\frac{1}{(1 - 2\beta)\sigma}\right)^{\frac{\ln(\varepsilon/C_1)}{\ln(\sigma)}}\right) &= \Oc\left(\left(\frac{1}{\varepsilon}\right)^{\frac{\ln((1-2\beta)\sigma)}{\ln(\sigma)}}},
\end{array}
\end{equation*}
where the last equality follows from the fact that $\tau^{\alpha\log(s)} = s^{\alpha\log(\tau)}$.
\end{proof}

%\section{LO Complexity Trade-off Between Two Stages}\label{apdx:subsec:trade_off}
%We provide the detail analysis of the LO complexity trade-off between two stages as stated in Subsection~\ref{subsec:trade_off}.
%Let us choose $\delta := \varepsilon$.
%Then, the number of iterations in the damped-step stage stated in Theorem~\ref{tm:finite_step_phase1} is $K = \frac{f(x^{0}) - f(x^{\star})}{\varepsilon\omega\left(\frac{1 - 2C_1}{C_1}h^{-1}(\beta)\right)} = \frac{R}{\varepsilon} = \BigO{\frac{1}{\varepsilon}}$, where $R := \frac{f(x^{0}) - f(x^{\star})}{\omega\left(\frac{1 - 2C_1}{C_1}h^{-1}(\beta)\right)}$.
%In this case, we have $(1-\delta)^K = (1-\varepsilon)^{\frac{R}{\varepsilon}} = \BigO{e^R}$ for sufficiently small $\varepsilon$.
%Consequently, by Theorem~\ref{tm:finite_step_phase1}, the total LO calls of the damped-step stage can be bounded by
%\begin{equation*}
%\Tc_1 := \frac{6D_{\Xc}^2\lambda_{\max}(\nabla^2 f(x^0))}{(C_1h^{-1}(\beta))^2}\frac{1 - (1-\delta)^{K+1}}{\delta(1-\delta)^K} \leq \frac{6D_{\Xc}^2\lambda_{\max}(\nabla^2 f(x^0))}{(C_1h^{-1}(\beta))^2}\cdot \frac{1}{\varepsilon \BigO{e^R}} = \BigO{\frac{1}{\varepsilon}}.
%\end{equation*}
%Hence, in terms of complexity order depending on the accuracy level $\varepsilon$, the total LO calls in the full-step stage is $\BigO{\varepsilon^{-2\nu}}$ dominates the one $\Tc_1 = \BigO{\varepsilon^{-1}}$ in the damped-step stage.

\bibliographystyle{plain}
%\bibliography{sections/my_bibtex}

% \input{sections/refs_supp}

%\section*{Acknowledgements}

%\bibliography{example_paper}

\end{document}